\documentclass[reqno]{amsart}
\usepackage{fkmath}

\usepackage{parskip}

\usepackage{kana}

\newcommand{\np}[1]{\hspace{-.55em}〔#1〕\hspace{-.55em}}

\usepackage{booktabs}

\usepackage{float}

\usepackage{subcaption}

\usepackage{scalerel}

\usepackage{kana}

\usepackage[colorlinks=true]{hyperref}

\usepackage[capitalise]{cleveref}

\usepackage{changepage}

\usepackage{tipa}

\usepackage[foot]{amsaddr}

\numberwithin{equation}{section}

\newtheorem{theorem}{Theorem}[section]
\newtheorem{proposition}[theorem]{Proposition}
\newtheorem{lemma}[theorem]{Lemma}
\newtheorem{corollary}[theorem]{Corollary}

\newtheorem{conjecture}{Conjecture}

\theoremstyle{definition}
\newtheorem{definition}[theorem]{Definition}
\newtheorem{example}[theorem]{Example}

\newtheorem{question}{Question}

\theoremstyle{remark}
\newtheorem{remark}[theorem]{Remark}

\newcommand{\metric}[1]{\ensuremath{d_{\mathsmaller{#1}}}}

\newcommand{\metspace}[1]{\ensuremath{(#1, \metric{#1})}}

\newcommand{\topt}[1]{\ensuremath{\ms T_{\mathsmaller{#1}}}}

\newcommand{\topspace}[1]{\ensuremath{(#1, \topt{#1})}}

\title{Uniform Convergence and Knot Equivalence}
\date{January 10\textsuperscript{th}, 2021}
\author{Forest Kobayashi}
\email{fkobayashi@math.ubc.ca}
\dedicatory{Dedicated to my grandparents, Albert and Elizabeth
  Kobayashi.}

\begin{document}

\maketitle

\begin{abstract}
  Given a uniformly convergent sequence of ambient isotopies $\pn{\ms
    H_n}_{n\in\NN}$, bijectivity of the limit function $\ms H_\infty$
  together with a minor compactness condition guarantees that $\ms
  H_\infty$ is also an ambient isotopy. By offloading the uniform
  convergence hypothesis to a more diagrammatic condition, we obtain
  sufficient conditions for performing countably-many Reidemeister
  moves. We use this to construct examples of tame knots with
  countably-many crossings and discuss what distinguishes these from
  similar-looking wild curves.
\end{abstract}

\tableofcontents

\section{Introduction}\label{sec:introduction}
This work is based on results from the author's undergraduate thesis
\cite{KobayashiThesis}.

The goals of this document are to understand when one can apply
countable sequences of Reidemeister moves and preserve ambient isotopy
in the limit. The motivating examples are the following two curves.
\begin{figure}[H]
  \centering
  \begin{subfigure}[t]{.49\linewidth}
    \centering
    \includegraphics[width=.7\linewidth]{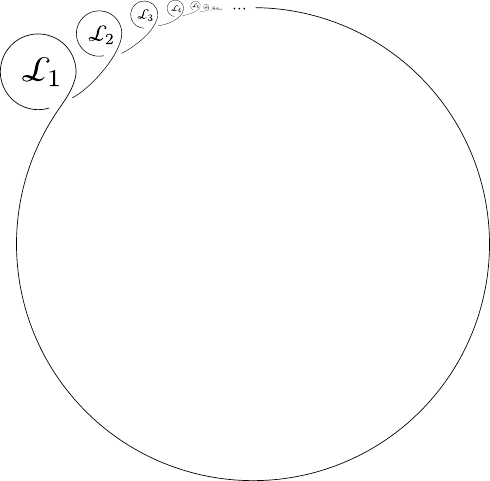}
    \caption{A wild-looking unknot.}
    \label{subfig:wild-looking-unknot}
  \end{subfigure}
  \begin{subfigure}[t]{.49\linewidth}
    \centering
    \includegraphics[width=.84\linewidth]{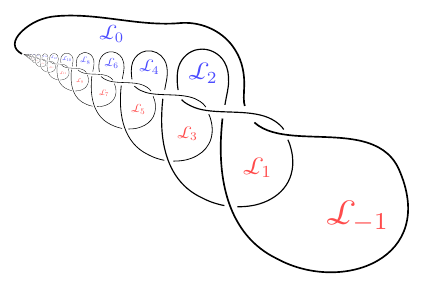}
    \caption{An unknotted-looking wild knot.}
    \label{subfig:remarkable-curve}
  \end{subfigure}
  \caption{Two scintillating curves.}
  \label{fig:scintillating-curves}
\end{figure}

\textbf{Tameness of \cref{subfig:wild-looking-unknot}.} As indicated
by the caption, the curve in \cref{subfig:wild-looking-unknot} is
tame. In fact, it is the unknot. An explicit ambient isotopy taking it
to the unknot can be constructed as follows. From time $t=0$ to
$t=\frac{1}{2}$, use a Reidemeister~I move to remove loop $\ms L_1$.
Then, from time $t=\frac{1}{2}$ to $t=\frac{3}{4}$, use a Reidemeister
I move to remove loop $\ms L_2$. So on and so forth, removing loop
$\ms L_n$ between time $t=1 - \frac{1}{2^{n-1}}$ and $t= 1 -
\frac{1}{2^n}$.

If one is careful about exactly how the Reidemeister I moves are
performed, then the result will be an ambient isotopy. We provide the
details in \cref{prop:wild-looking-unknot-tame}.

\textbf{Wildness of \cref{subfig:remarkable-curve}.} By contrast, the
curve in \cref{subfig:remarkable-curve} is wild, a result established
by Ralph Fox in \cite{Fox}. His argument uses techniques co-developed
with Emil Artin % \footnote{Pronunciation guide:
  % \ipa{"a\textsubarch{5}ti:n}}
in \cite{Fox-Artin}, namely, a sort of ``invariant'' for tameness of
arcs. We summarize the relevant result in
\cref{sec:tameness-invariant}.

\textbf{The Problem.} At first, the wildness of the curve in
\cref{subfig:remarkable-curve} can appear very counterintuitive. As
with the loops in \cref{subfig:wild-looking-unknot}, any \emph{finite}
number of the ``stitches'' in \cref{subfig:remarkable-curve} can be
safely removed using Reidemeister II moves. The procedure is as
follows. First, use a Reidemeister II move to move $\color{red} \ms
L_{-1}$ into $\color{red} \ms L_1$. Next, use another Reidemeister II
move to remove both $\color{red} \ms L_{-1}$ and $\color{blue} \ms
L_2$.
\begin{figure}[H]
  \centering
  \begin{subfigure}[t]{.490\linewidth}
    \centering
    \includegraphics[scale=1.75]{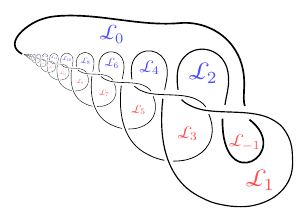}
    \caption{Moving $\color{red} \ms L_{-1}$ into $\color{red} \ms L_1$.}
  \end{subfigure}
  \begin{subfigure}[t]{.490\linewidth}
    \centering
    \includegraphics[scale=1.75]{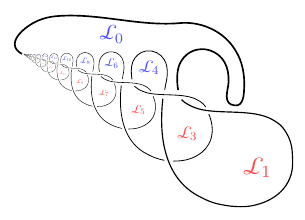}
    \caption{Removing $\color{red} \ms L_{-1}$ and $\color{blue} \ms
      L_2$.}
  \end{subfigure}
  \caption{Removing one stitch.}
\end{figure}
Note, this leaves us with a shrunk copy of the same diagram we started
with. Hence we can repeat the process any finite number of times. In
general, to perform step $n$ we use a Reidemeister II move to slide
loop $\color{red} \ms L_{2n-3}$ into loop $\color{red} \ms L_{2n-1}$,
then we remove $\color{red} \ms L_{2n - 3}$ and $\color{blue} \ms
L_{2n}$ by using another Reidemeister II move.

What's to stop us from using the same approach as in
\cref{subfig:wild-looking-unknot}, where we just performed step $n$
between $t=1-\frac{1}{2^{n-1}}$ and $t=1 - \frac{1}{2^n}$? Indeed, if
we choose our Reidemeister II moves carefully, it's possible to
guarantee continuity of the resulting limit function.

The issue is with bijectivity on the ambient space. As we show in
\cref{sec:does-not-apply}, there is no way to choose our sequence of
Reidemeister II moves without accidentally dragging points from the
ambient space down to the wild point in the limit. The proof is a
simple geometric argument but it can be easy to overlook.
% { \color{blue}
%   As we will see, this suggests that there is a sense in
%   which wild knots do not always have to be ``caught'' on themselves.
%   Rather, sometimes the problem is that they are just so
%   pathologically embedded that untying them is too involved to do
%   without tearing or breaking the ambient space.}

\textbf{The Upshot.} To clarify the situation, in
\cref{thm:vks-ambient-homeomorphism,thm:vks-ambient-isotopy} we
generalize the strategy we described for
\cref{subfig:wild-looking-unknot} and make it rigorous.
\cref{thm:vks-ambient-homeomorphism} uses the language of
\emph{ambient homeomorphisms} while \cref{thm:vks-ambient-isotopy}
uses the language of \emph{ambient isotopies}; other than that, they
are equivalent.

The layout of the rest of the paper is as follows:
\begin{enumerate}[label=(\S\arabic*)] \setcounter{enumi}{1}
  \item In \cref{sec:background}, we go over the definitions of knots,
    ambient homeomorphisms, ambient isotopies, PL-ness, and
    tameness/wildness. It might be helpful to review these concepts
    since we will be working directly with equivalence in this paper
    instead of by proxy through Reidemeister's Theorem. We also give
    the definition of uniform convergence and recall two elementary
    results from first courses in Analysis and Topology, respectively;
    the main theorems will follow from these.
  \item In \cref{sec:uniform-convergence}, we give the definition of
    uniform convergence and use it to build up our main results
    (\cref{thm:vks-ambient-homeomorphism,thm:vks-ambient-isotopy}). In
    most cases \cref{thm:vks-ambient-homeomorphism} tends to be more
    ergonomic than \cref{thm:vks-ambient-isotopy}, but we will
    generally employ \cref{thm:vks-ambient-isotopy} because it seems
    the language of ambient isotopy is more ubiquitous than that of
    ambient homeomorphism.
  \item In \cref{sec:various-applications}, we apply
    \cref{thm:vks-ambient-isotopy} to various curves with
    countably-many crossings.
  \item In \cref{sec:does-not-apply}, we show what goes wrong if we
    try to apply \cref{thm:vks-ambient-isotopy} to two noteworthy
    examples; the second is the wild curve from
    \cref{subfig:remarkable-curve}.
\end{enumerate}

Finally, some notation.
\begin{table}[H]
  \centering
  \footnotesize
  \begin{tabular}{ccl}
    \toprule
    Symbol && Interpretation \\ \midrule
    $f$ && Embedding and/or knot (usually). \\
    \midrule
    $h$ && Homeomorphism. \\
    $\msf{h}$ && PL Homeomorphism. \\
    $\ms h$ && Homeomorphism constructed by composing other
               homeomorphisms. \\ \midrule
    $H$ && Ambient isotopy. \\
    $\msf{H}$ && PL Ambient isotopy. \\
    $\ms H$ && Ambient isotopy constructed by iteratively gluing other
               ambient isotopies. \\\midrule
    $\metspace{X}$ && Metric space $X$ with metric $\metric{X}$. \\
    $\topspace{X}$ && Topological space $X$ with topology $\ms T_{X}$.
    \\\midrule
    $\comp_{k=1}^n f_k$ && The composite function $f_n \circ f_{n-1} \circ
                           \cdots \circ f_2 \circ f_1$. \\
    $k$ && Generally reserved for an ``free'' index (e.g., $k
           \in \set{1, \ldots, n}$ above).\\
    $n$ && Generally reserved for a ``particular'' index (e.g., $n$ is
           the ``stop'' index above).\\\midrule
    $\pn{f_k}_{k=1}^\infty$ && A sequence $f_1, f_2, \ldots,
                               f_k, \ldots$. \\
    $f_k \to f$ && The $f_k$ converge to $f$ pointwise. \\
    $f_k \uconv f$ && The $f_k$ converge to $f$ uniformly. Note, some
                      authors use $\rightrightarrows$ instead of
                      $\uconv$. \\\midrule
    $A^\circ$ && Interior of $A$. \\
    $\ol{A}$ && Closure of $A$. \\
    $A^c$ && $X \setminus A$ (whenever $X$ is understood) \\
    $A_1 \setminus A_2$ && Set difference of $A_1$ and $A_2$.\\
    $A_1 \sqcup A_2$ && Disjoint union of $A_1$ and $A_2$.\\\midrule
    \np{\ldots} && Used to indicate parsing order in
                   grammatically-ambiguous sentences. \\
    $\cmark$ && Indicates the completion of a case in a proof by
                casework.\\\bottomrule
  \end{tabular}
\end{table}

\section{Background}\label{sec:background}
Today we will be working with knot equivalence directly instead of
making appeals to \emph{Reidemeister's theorem}. This is because we're
interested in knots that might be \emph{wild}
(\cref{def:tame-and-wild}), but Reidemeister's theorem assumes
\emph{tameness} (also \cref{def:tame-and-wild}). We begin with a
reminder of some fundamental definitions.

\subsection{Fundamental Definitions}\label{subsec:fundamental-definitions}

Let $\topspace{X}$, $\topspace{Y}$ be topological spaces. An
\emph{embedding} of $X$ into $Y$ is a function $f \colon X \to Y$ such
that restricting the codomain of $f$ to $f(X)$ gives us a
homeomorphism $\widetilde{f} \colon X \to f(X)$.
\begin{figure}[H]
  \centering
  \includegraphics[scale=.25]{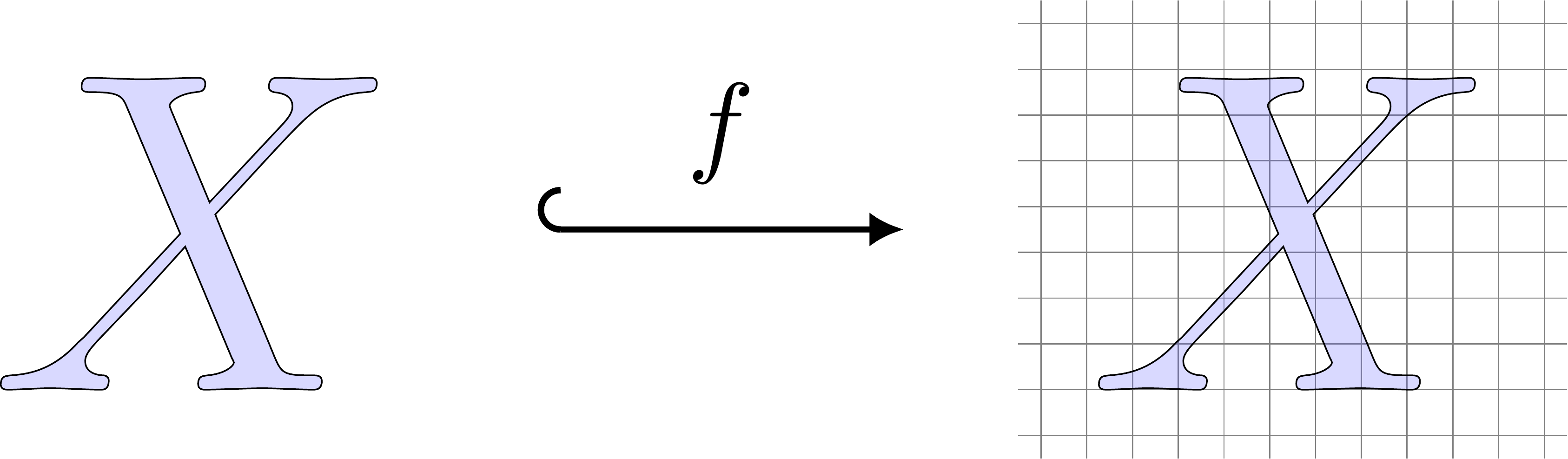}
  \caption{Example of embedding an $X$ into $\RR^2$.}
\end{figure}
Since embeddings must be injective, some authors choose to denote them
by $f \colon X \into Y$. Here we call $Y$ the \emph{ambient space} and
refer to $f(X)$ as \emph{$X$ embedded by $f$ in $Y$}.

Two embeddings $f_1$, $f_2 \colon X \to Y$ are said to be \emph{equivalent}
(denoted $f_1 \cong f_2$) if there exists a homeomorphism $h \colon Y \to
Y$ such that for all $x \in X$,
\[
  (h \circ f_1)(x) = f_2(x).
\]
Since $h$ is a homeomorphism on the \emph{ambient} space, we refer to
it as an \emph{ambient homeomorphism}.

\begin{remark}
  This definition requires pointwise equality for $(h \circ f_1)$,
  $f_2$. In general, this is stronger than requiring $h(f_1(X)) =
  f_2(X)$ as sets.\footnote{The correspondence holds for \emph{tame}
    knots and certain {everywhere-wild} knots. An example of a knot
    for which it fails is \cref{subfig:remarkable-curve}. The idea is
    that the ambient homeomorphism can only send wild points to other
    wild points, and this makes it impossible to pull the strand
    ``through'' the wild point.} The interested reader should see
  \cite{Bothe1981, Shilepsky} for more.
\end{remark}
Geometrically, we think of $h$ as ``deforming'' the ambient space to
take $f_1(X)$ to $f_2(X)$.
\begin{figure}[H]
  \centering
  \includegraphics[scale=.25]{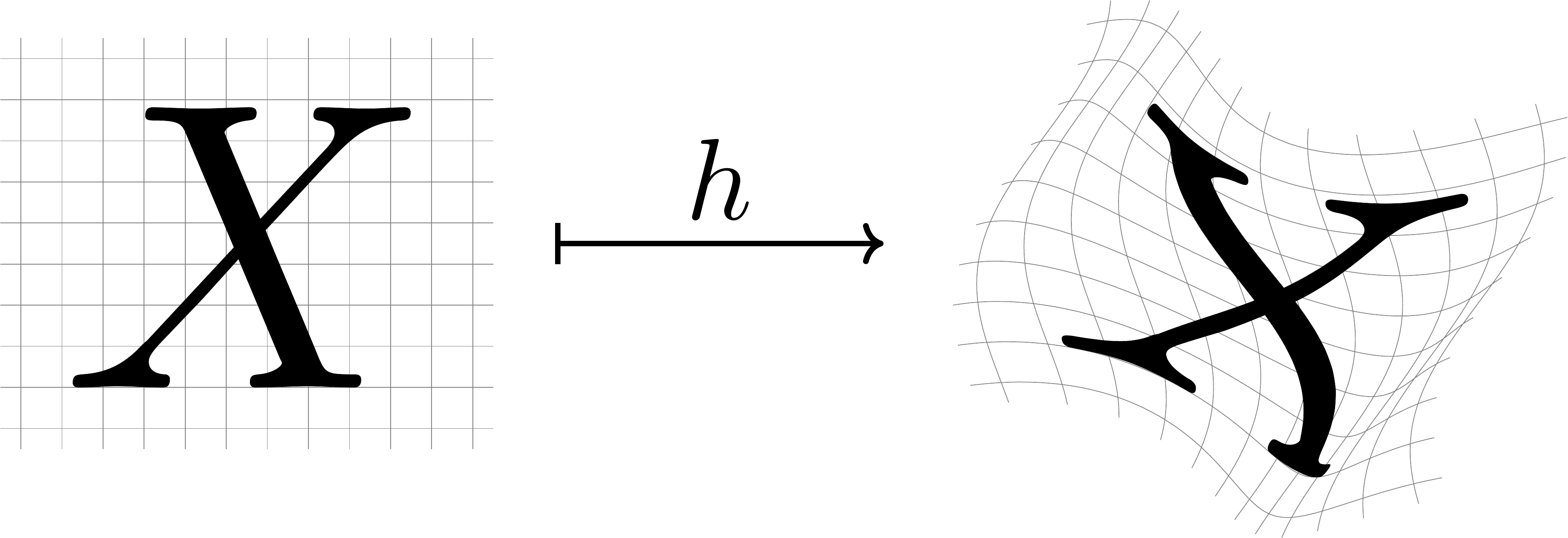}
  \caption{An example $h$ taking $f_1(X)$ to a distorted version
    representing $f_2(X)$.}
  \label{fig:example-homeomorphism}
\end{figure}
A \emph{knot} is an embedding $f \colon S^1 \into \RR^3$. One should note
that some authors take the codomain to be $S^3$ instead of $\RR^3$
because $S^3$ is compact. Our proofs today only require that $Y$ be a
metric space, hence we are free to choose either option. We could even
choose a thickened orientable surface in order to work with virtual
knots. However, we will do neither of these things and instead choose
$Y = \RR^3$ because it is easier to represent graphically.

For embeddings in $\RR^3$, defining equivalence through ambient
homeomorphisms is equivalent to defining equivalence with
\emph{ambient isotopy} (defined below). We refer to this fact as the
\emph{equivalence of equivalences}. For further discussion (as well as
a list of references about the correspondence in each of the
\textbf{PL}, $\mb{C^\infty}$, and \textbf{Topological} categories),
see \cite{KobayashiThesis}, particularly \S 6.3.
\begin{definition}[Ambient Isotopy]
  Let $\topspace{X}$, $\topspace{Y}$ be topological spaces. Let $f_1$,
  ${f_2} \colon X \into Y$ be embeddings. Then a function $H \colon
  [0,1] \times Y \to Y$ is called an \emph{ambient isotopy} iff
  \begin{enumerate}
    \item $H$ is continuous,
    \item $H(0, \cdot)$ is the identity on $Y$,
    \item For all $t \in [0,1]$, the function $H(t, \cdot) \colon Y
      \to Y$ is a homeomorphism, and
    \item For all $x \in X$, we have
      \[
      (H(1, \cdot) \circ {f_1})(x) = {f_2}(x).
      \]
  \end{enumerate}
  We often refer to $H$ as an \emph{ambient isotopy from
    ${f_1}$ to ${f_2}$}.
\end{definition}
Note that $H(1,\cdot)$ is an ambient homeomorphism from
${% \color{red}
  f_1}$ to ${% \color{blue}
  f_2}$. We can think of the $t$ variable as describing a ``time''
parameter in a movie connecting $H(0,\cdot)$ to $H(1,\cdot)$.
\begin{figure}[H]
  \centering
  \includegraphics[width=\textwidth]{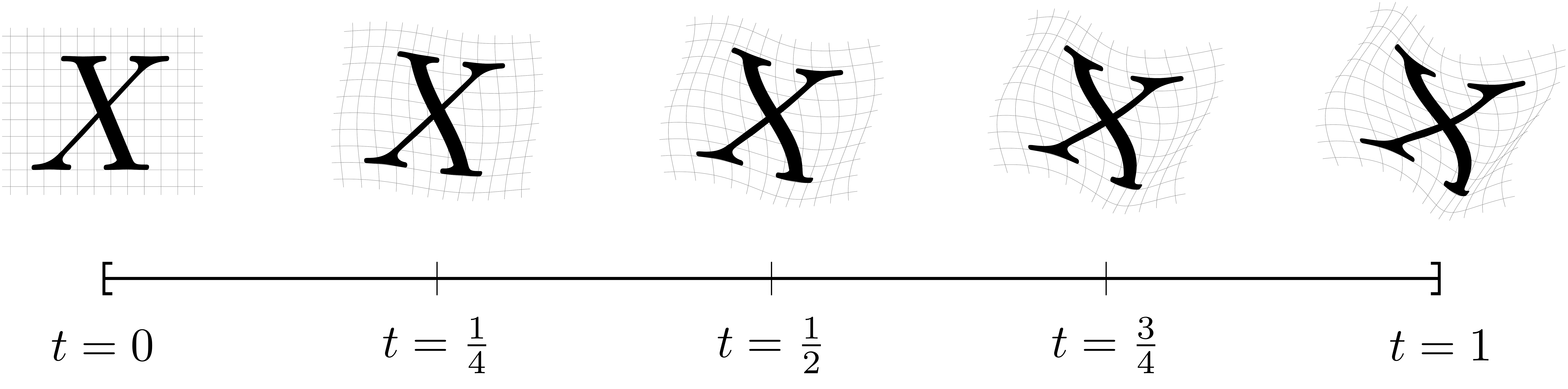}
  \caption{5 freeze-frames from an ambient isotopy where $H(1,\cdot)$
    is the $h$ in \cref{fig:example-homeomorphism}}
\end{figure}
We will prove our results today for both ambient homeomorphisms and
ambient isotopies. Although this is not strictly necessary in light of
the equivalence of equivalences, we have chosen to include both
arguments to illustrate the additional step required for working with
ambient isotopy.

\subsection{Tameness and Wildness}
Oftentimes, knot theory is restricted to the study of \emph{tame}
knots, which we define in a moment. We think of tame knots as being
well-behaved because they belong to equivalence classes of knots that
have representative elements that can described with finite
information, namely \emph{polygonal} or \emph{PL knots}.\footnote{PL
  stands for Piecewise Linear. For more on PL topology, the reader
  might look at J.L.\ Bryant's \emph{Piecewise Linear Topology}. An
  online version can be found at
  \url{https://www.maths.ed.ac.uk/~v1ranick/papers/pltop.pdf}} This is
the loose intuition underpinning Reidemeister's theorem.

\begin{definition}[Polygonal Knot]
  A \emph{polygonal knot} is a knot that is comprised of a finite
  union of straight line segments.
\end{definition}

\begin{figure}[H]
  \centering
  \includegraphics[scale=.4]{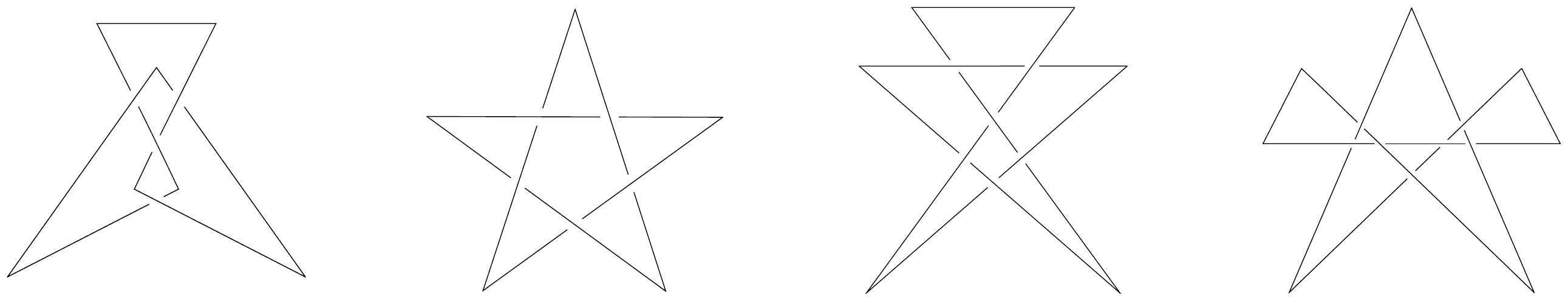}
  \caption{Examples of some polygonal knots}
\end{figure}

\begin{definition}[Tame \& Wild Knots]\label{def:tame-and-wild}
  A \emph{tame knot} is a knot that is ambient isotopic to a polygonal
  knot. A \emph{wild knot} is a knot that is not tame.
\end{definition}
\begin{figure}[H]
  \centering
  \includegraphics{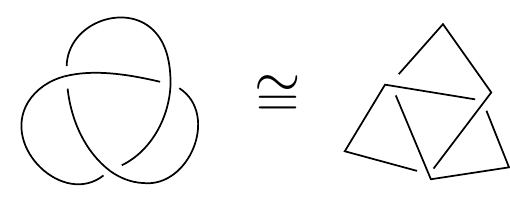}
  \caption{Example of a tame knot}
\end{figure}
\begin{remark}
  There are many other common definitions for tameness and wildness. A
  discussion of these definitions (and some of the equivalences) can
  be found in \cite{KobayashiThesis}.
\end{remark}

\subsection{Uniform Convergence}
Finally, we recall the definition of \emph{uniform convergence}.

\begin{definition}[Uniform Convergence]
  Let $\metspace{X}$, $\metspace{Y}$ be metric spaces, and consider a
  sequence of functions $f_n \colon X \to Y$. Suppose that the $f_n$
  converge pointwise to some $f \colon X \to Y$. Then we say the $f_n$
  \emph{converge to $f$ uniformly} iff for all $\varepsilon > 0$,
  there exists $n_0 \in \NN$ such that for all $x \in X$, $n > n_0$
  implies
  \[
    \metric{Y}(f_n(x), f(x)) < \varepsilon.
  \]
\end{definition}

We typically denote uniform convergence by $f_n \uconv f$.
Geometrically, we think of this in terms of pictures like the
following:
\begin{figure}[H]
  \centering
  \includegraphics[scale=.5]{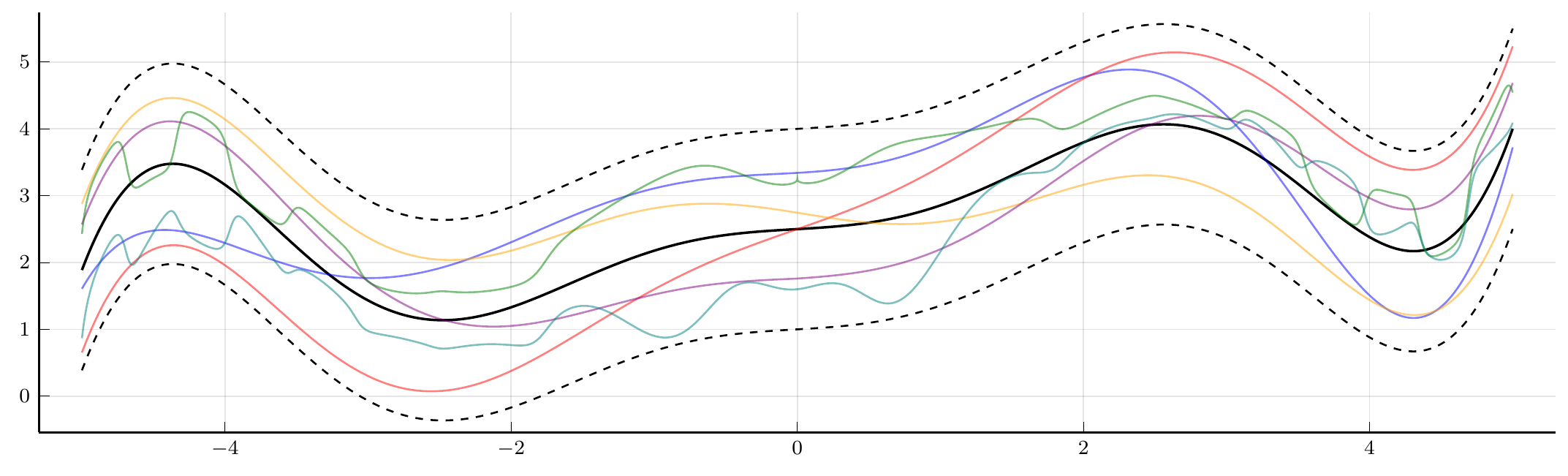}
  \caption{Example of some $f_n \colon \RR \to \RR$ satisfying $\norm{f_n -
      f}_{\infty} < \varepsilon$. Dashed lines indicate $f(x) \pm
    \varepsilon$.}
\end{figure}
\begin{adjustwidth}{1em}{}
  \begin{leftbar}\vspace{-.5em}
    \begin{remark}
      It's worth noting that, as with the definition of ambient
      homeomorphism, this definition is generally stronger than
      requiring convergence of the images as sets. In fact, we can
      construct simple examples of $f_n \colon [0,1] \into \RR^3$ such
      that for all $n$, $m\in\NN$, $f_n([0,1]) = f_m([0,1])$, but the
      $f_n$ do not even converge pointwise. One way to achieve this is
      to alternate between two different parameterizations of the same
      curve.
      \begin{figure}[H]
        \centering
        \begin{subfigure}[t]{.49\linewidth}
          \centering
          \includegraphics[scale=.4]{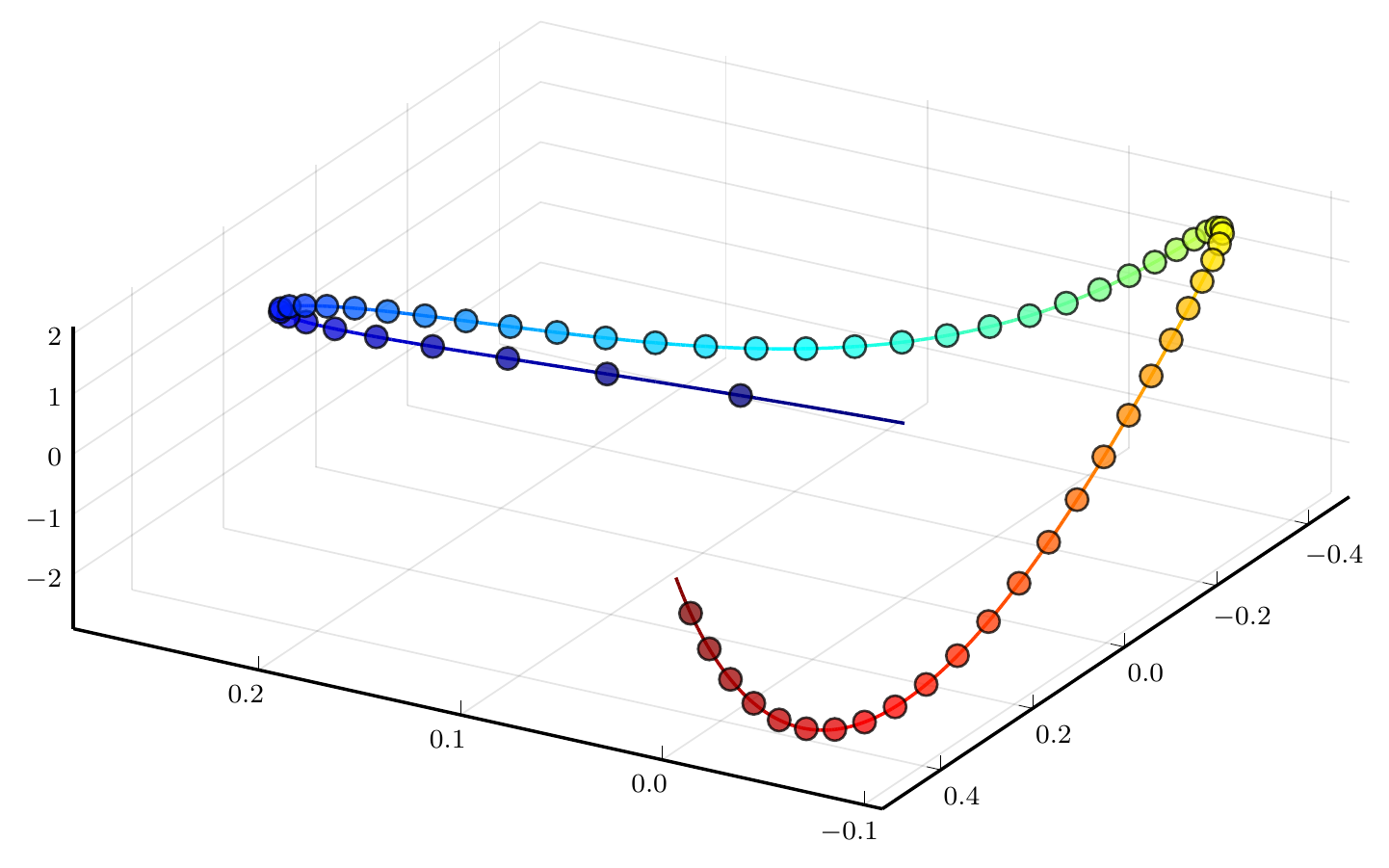}
          \caption{One parameterization\ldots}
        \end{subfigure}
        \begin{subfigure}[t]{.49\linewidth}
          \centering
          \includegraphics[scale=.4]{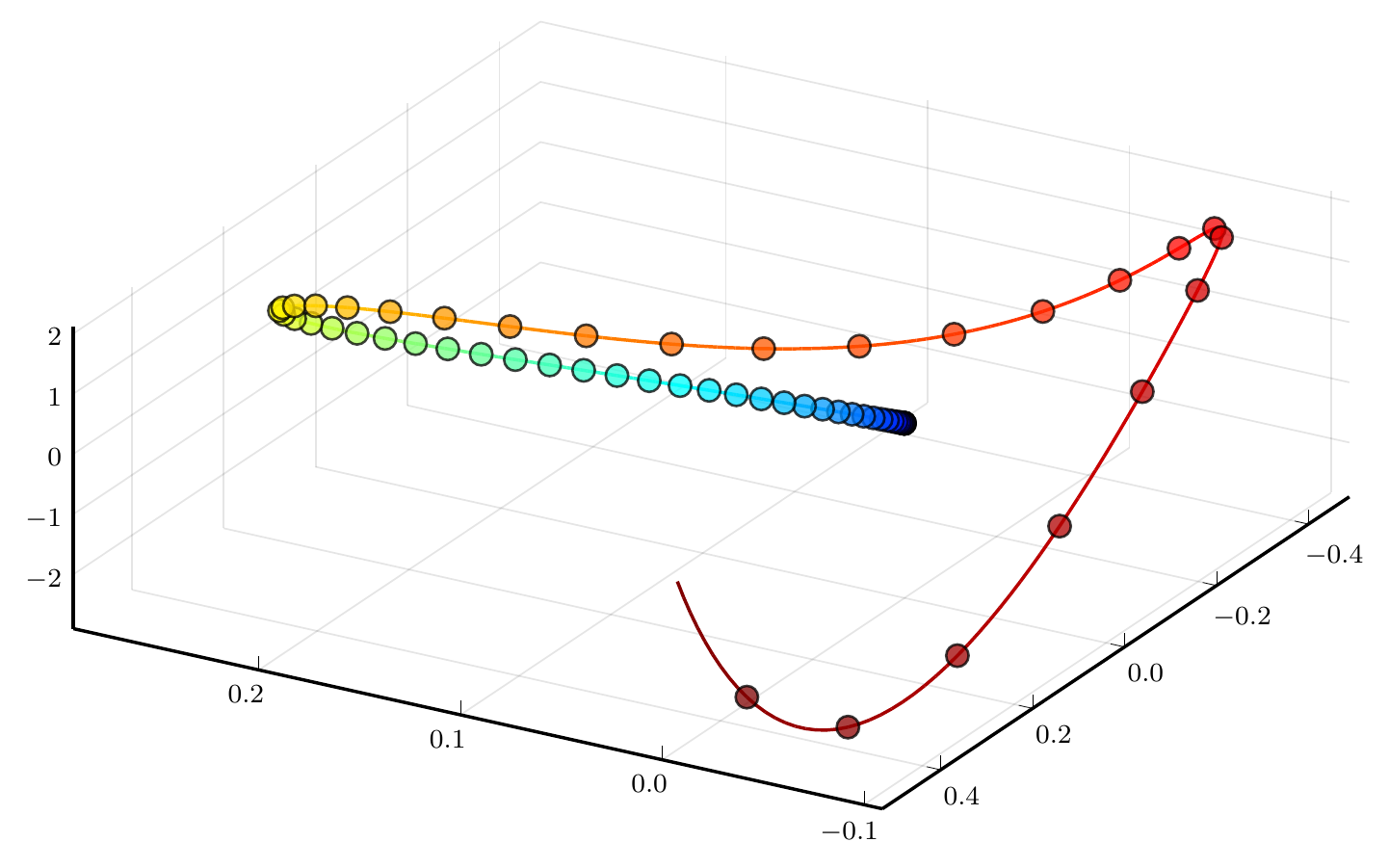}
          \caption{\ldots And another.}
        \end{subfigure}
        \caption{Two parameterizations of the same curve.}
      \end{figure}
      In the figures above, the colors correspond to value of $t \in
      [0,1]$ that yields the given point under $f_n$. In order for
      $f_n \uconv f$, not only do the curves have to take on the same
      ``shape,'' but also every point of a given color in $f_n([0,1])$
      has to be less than $\varepsilon$ away from the corresponding
      point in $f([0,1])$.
    \end{remark}
  \end{leftbar}
\end{adjustwidth}
We recall two results from first courses in Analysis and Topology,
respectively.
\begin{proposition}\label{prop:uniform-convergence-continuity}
  Let $(X, \metric{X})$, $(Y, \metric{Y})$ be metric spaces. For each
  $k \in \NN$, let $f_k \colon X \to Y$ be continuous. Suppose that
  there exists $f \colon X \to Y$ such that $f_k \uconv f$. Then $f$
  is continuous.
\end{proposition}

\begin{proposition}\label{prop:compact-hausdorff-homeomorphism}
  Let $\topspace{X}$, $\topspace{Y}$ be topological spaces. Suppose
  that $X$ is compact and $Y$ is Hausdorff. Now suppose that $f \colon X
  \to Y$ is bijective and continuous. Then $f$ is a homeomorphism.
\end{proposition}
% A typical proof is as follows.
% \begin{proof}
%   Bijectivity and continuity are given, it remains to show that
%   $f^{-1}$ is continuous. To that end, since $f$ is a bijection, it
%   suffices to show that $f$ is a closed map.

%   Let $V \subseteq X$ be an arbitrary closed set. We want to show
%   $f(V)$ is closed. $X$ is compact, so $V$ is a closed subset of a
%   compact set and hence $V$ is compact. Then $f(V)$ is a continuous
%   image of a compact set, so we have $f(V)$ is compact. Finally,
%   because $Y$ is Hausdorff, $f(V)$ is compact implies $f(V)$ is closed
%   in $Y$. So $f$ is closed, as desired.

%   It follows that $f$ is a homeomorphism.
%   \renewcommand{\qedsymbol}{$\square$}
% \end{proof}
We are now ready to begin our discussion of these results in the
context of knots.

\section{Uniform Convergence \& Knots}\label{sec:uniform-convergence}
These two propositions give us the following simple result.
\begin{corollary}\label{cor:uniformly-convergent-homeomorphism}
  Let $(X, \metric{X})$, $(Y, \metric{Y})$ be metric spaces, and
  suppose that $X$ is compact. For each $k \in \NN$, let $f_k \colon X
  \into Y$ be an embedding. Suppose that the $f_k$ converge uniformly
  to some $f \colon X \to Y$. Then if $f$ is injective, it follows
  that $f$ is an embedding.
\end{corollary}
\begin{proof}
  By hypothesis, for all $k \in \NN$ we have $f_k$ is an embedding and
  hence continuous. Since we also have $f_k \uconv f$,
  \cref{prop:uniform-convergence-continuity} implies $f$ is
  continuous.

  $\metspace{Y}$ is a metric space and thus Hausdorff. So $f(X)$ with
  the subspace topology is Hausdorff. Now, $f$ is injective and thus
  $f$ is a bijection between $X$ and $f(X)$. By
  \cref{prop:compact-hausdorff-homeomorphism} it follows that $f$ is a
  homeomorphism between $X$ and $f(X)$. Thus $f$ is an embedding.
\end{proof}
Among other things, \cref{cor:uniformly-convergent-homeomorphism} can
be used to construct fractal-like knot diagrams.
\begin{example} \label{ex:koch-knot}
  Consider the knot constructed by the following iterative procedure,
  loosely inspired by the Koch Snowflake:
  \begin{figure}[H]
    \centering
    \begin{subfigure}{.32\linewidth}
      \centering
      \includegraphics[scale=.75]{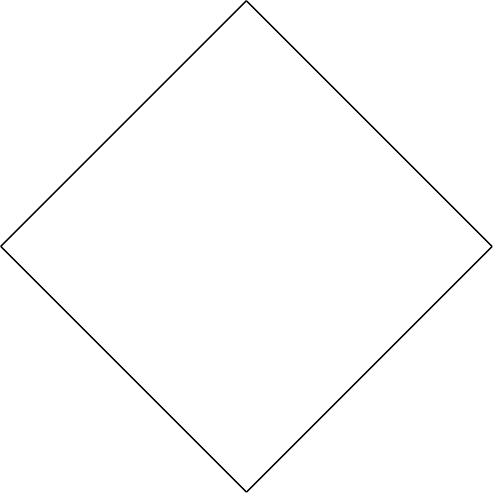}
      \caption{$f_1$}
    \end{subfigure}
    \begin{subfigure}{.32\linewidth}
      \centering
      \includegraphics[scale=.75]{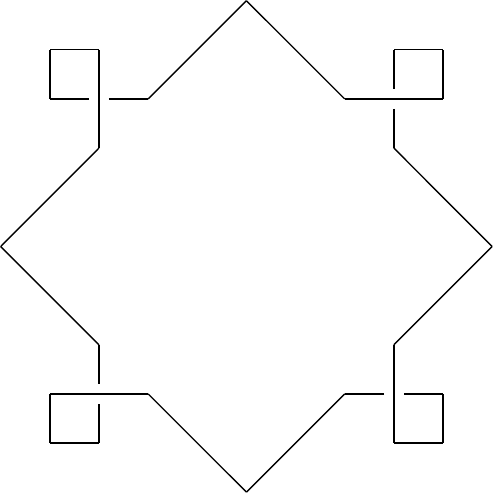}
      \caption{$f_2$}
    \end{subfigure}
    \begin{subfigure}{.32\linewidth}
      \centering
      \includegraphics[scale=.75]{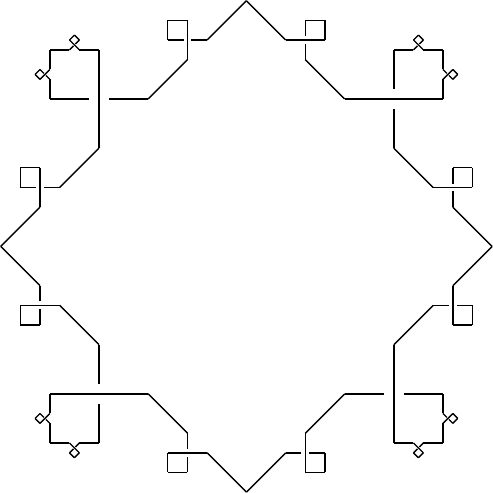}
      \caption{$f_3$}
    \end{subfigure}
  \end{figure}
  \begin{figure}[H]\ContinuedFloat
    \centering
    \begin{subfigure}{.32\linewidth}
      \centering
      \includegraphics[scale=.75]{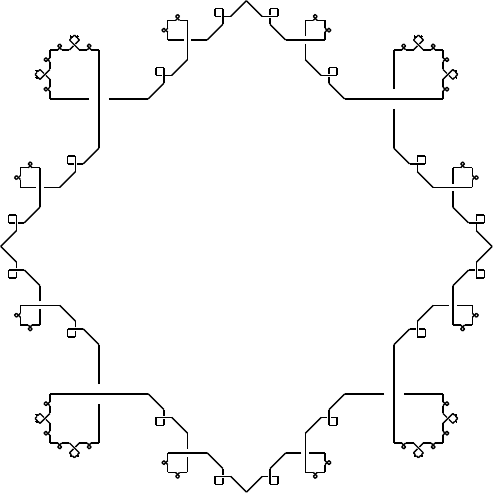}
      \caption{$f_4$}
    \end{subfigure}
    \begin{subfigure}{.32\linewidth}
      \centering
      \includegraphics[scale=.75]{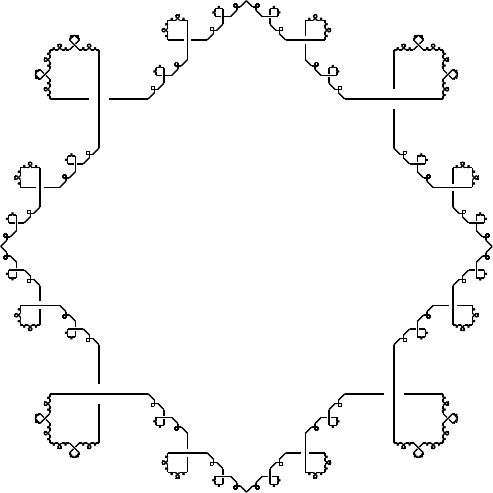}
      \caption{$f_5$}
    \end{subfigure}
    \begin{subfigure}{.32\linewidth}
      \centering
      \includegraphics[scale=.75]{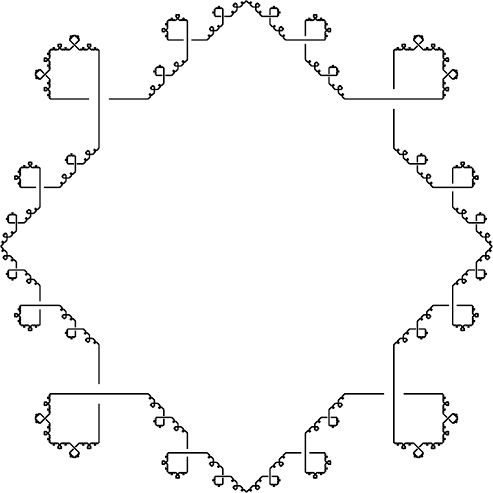}
      \caption{$f_6$}
    \end{subfigure}
    \caption{A ``snowflake'' knot.}
    \label{fig:snowflake-knot}
  \end{figure}
  One can show that with the proper choice of parameterizations for
  the $f_k$, \cref{cor:uniformly-convergent-homeomorphism} guarantees
  the limit function $f_\infty = \lim_{k\to\infty} f_k$ is an
  embedding. The main challenge is explicitly proving injectivity;
  this can be done as long as the ``shrink'' factor for the twists is
  sufficiently small.
\end{example}

\subsection{Iteratively Constructing (Ambient) Homeomorphisms}
For the rest of this document we will be primarily interested in a
special case of \cref{cor:uniformly-convergent-homeomorphism}. Namely,
when $X = Y$, the $f_k$'s become homeomorphisms from $X$ to itself.
When $X$ is a compact subset of $\RR^3$ this will give us a way to
construct ambient homeomorphisms (and later, ambient isotopies) by
composing countably-many Reidemeister moves.

To that end we repackage \cref{cor:uniformly-convergent-homeomorphism}
into a form that is more ergonomic when working with this special
case. In particular, we'll write the limit function in terms of a
\emph{composition} of homeomorphisms, which is more in line with our
intuition of applying multiple Reidemeister moves in succession. This
will also allow us to offload the uniform convergence requirement to
one of acting on a shrinking collection of neighborhoods, which is
easier to interpret in terms of knot diagrams. Note, the following
theorem is true for homeomorphisms in general, but we are only
interested in applying it to ambient homeomorphisms.

\begin{theorem}\label{thm:vks-ambient-homeomorphism}
  $(Y, \metric{Y})$ be a metric space. For all $k \in \NN$, let $h_k
  \colon Y \to Y$ be a homeomorphism, and for all $n \in \NN$, define
  \[
    \ms h_n = \comp_{k=1}^n h_k = (h_n \circ h_{n-1} \circ \cdots
    \circ h_{2} \circ h_1).
  \]
  For each $k$ let $V_k \subseteq Y$ such that $h_k$ is identity on
  $V_k^c$. Then provided
  \begin{enumerate}
    \item The $V_k$ satisfy
      \[
      \lim_{n\to\infty} \diam\pn[bigg]{\bigcup_{k=n}^\infty V_k} = 0,
      \]
      \label{cond:vk-vanish}
    \item There exists a compact $A \subseteq Y$ such that
      \[
      \pn[bigg]{\bigcup_{k=1}^\infty V_k} \subseteq A^\circ ,
      \]
      and \label{cond:vk-compactly-contained}
    \item $\ms h_\infty$ defined by
      \begin{align*}
        \ms h_\infty
        &= \lim_{n\to\infty} \ms h_n
      \end{align*}
      exists and is bijective,
  \end{enumerate}
  then $\ms h_\infty$ is a homeomorphism.
\end{theorem}
Before continuing to the proof, we make some remarks about the
statement.
\begin{remark}
  The hypothesis that the limit $\ms h_\infty$ exists is superfluous
  as it is implied by conditions (\ref{cond:vk-vanish}) and
  (\ref{cond:vk-compactly-contained}).
\end{remark}
\begin{remark}
  One can replace the somewhat-technical conditions on the $V_k$ with
  simpler ones. E.g., conditions (\ref{cond:vk-vanish}) and
  (\ref{cond:vk-compactly-contained}) could be substituted with
  \begin{enumerate}
    \item Requiring
      \[
      \cdots \subseteq V_{k+1} \subseteq V_k \subseteq \cdots
      \subseteq V_1 \subseteq A^\circ
      \]
      and
    \item $\lim_{k\to\infty} \diam(V_k) = 0$.
  \end{enumerate}
  Another option would be to replace condition (\ref{cond:vk-vanish})
  with something like ``$\diam(\limsup V_n) = 0$.'' In any case, we
  avoided simplifications like these in an effort to make
  correspondence with the hypotheses of
  \cref{cor:uniformly-convergent-homeomorphism} more direct.
\end{remark}
% \begin{remark}
%   The result also holds if we can partition the $V_k$ into finitely
%   many disjoint subsequences such that the conditions hold on each
%   subsequence.
% \end{remark}
\begin{remark}\label{rem:harmonic-series-BAD}
  We require $\lim_{n\to\infty} \diam\pn{\bigcup_{k=n}^\infty
    V_k} = 0$ instead of $\lim_{n\to\infty} \diam\pn{V_k} = 0$ to
  avoid situations where the $V_k$ have diameters like $\frac{1}{k}$.
  If we were to allow cases like these the divergence of the harmonic
  series would cause problems.
\end{remark}
\begin{remark}
  Though perhaps tempting, it is not sufficient to do away with the
  conditions on the $V_k$ by requiring something like ``$h_k \uconv
  1_Y$.'' As a counterexample, consider $Y = [0,1]$ with the standard
  metric on $\RR$. For all $k \in \NN$, define
  \[
    h_k = x^{(k+1)/k}.
  \]
  Then $h_k \uconv 1_{\bk{0,1}}$. But note, $\ms h_k = x^k$, and thus
  \[
    \ms h_\infty(x) =
    \begin{cases}
      0 & x \in [0, 1) \\
      1 & x = 1
    \end{cases}
  \]
  which is not a homeomorphism.
\end{remark}
\begin{proof}
  We will employ the gluing lemma. To that end, we need to partition
  $Y$ into two closed sets and show that $\ms h_\infty$ is a
  homeomorphism on both. A natural choice is to consider $\ol{A^c}$
  and $A$. Note, the compactness of the latter will allow us to appeal
  to \cref{cor:uniformly-convergent-homeomorphism}.
  \begin{figure}[H]
    \centering
    \includegraphics[width=.7\linewidth]{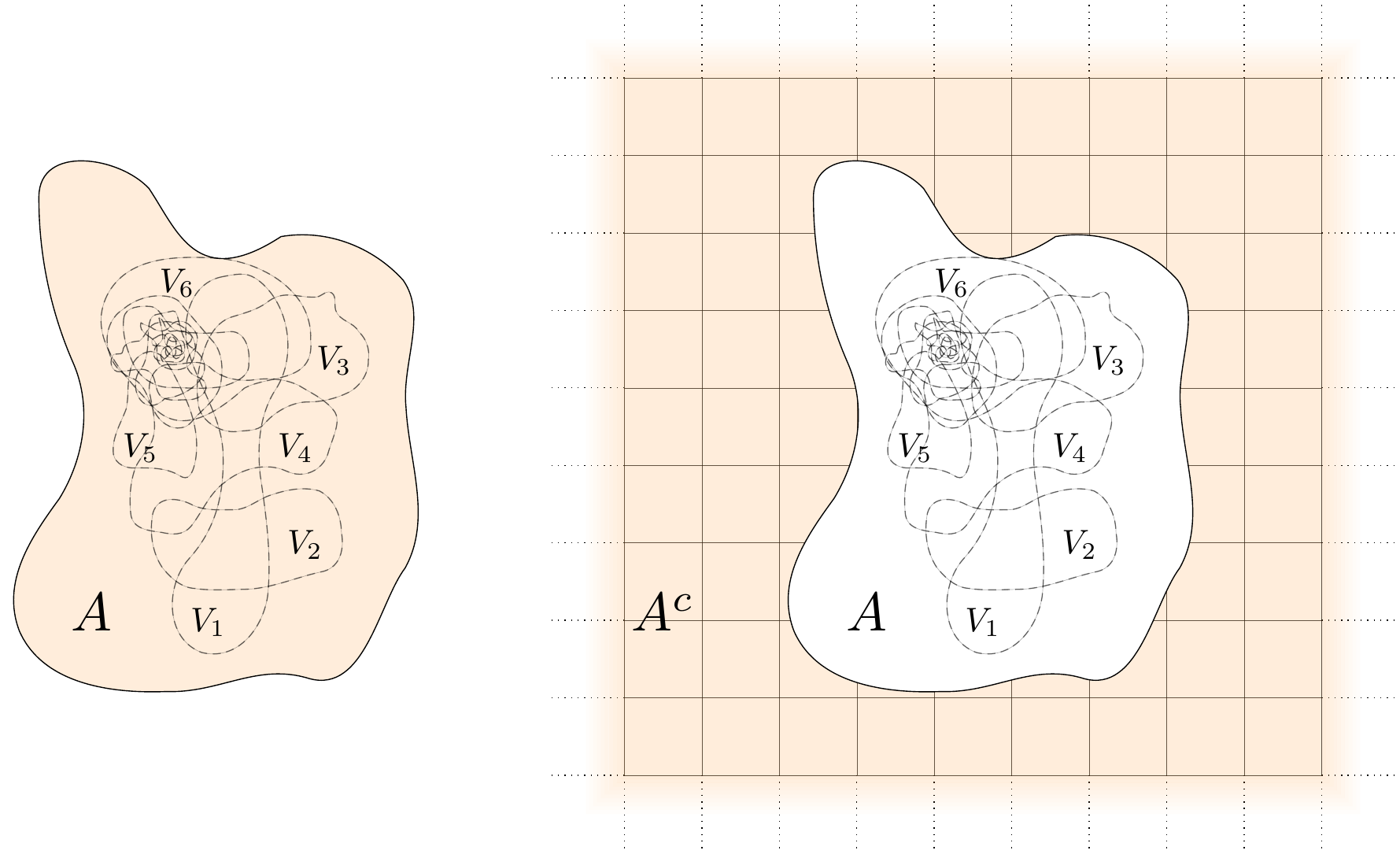}
    \caption{Example $A$ shown shaded on the left; example $\ol{A^c}$
      shown shaded on the right.}
  \end{figure}
  We examine these two sets separately.
  \begin{enumerate}
    \item (On $\ol{A^c}$): By construction, each $h_k$ is
      identity on $V_k^c$. Since each $V_k \subseteq A^\circ$, it follows
      $\ms h_\infty$ is identity (and hence a homeomorphism) on
      $(A^\circ)^c = \ol{A^c}$.
    \item (On $A$): Now, we show that $\ms h_\infty$ is a
      homeomorphism on $A$. By
      \cref{cor:uniformly-convergent-homeomorphism}, because $\ms
      h_\infty$ was assumed to be bijective it suffices to
      show that the restrictions $\ms h_k |_{A}$ converge uniformly to
      $\ms h_\infty |_{A}$. We will suppress writing the $|_{A}$
      for now because it clutters the notation too much.

      Let $\varepsilon > 0$ be given. Recall that by hypothesis, we
      have
      \[
      \lim_{n\to\infty} \diam\pn[bigg]{\bigcup_{k=n}^\infty V_k} = 0,
      \]
      hence there exists $n_0 \in \NN$ such that
      \[
      \diam\pn[bigg]{\bigcup_{k> n_0}^\infty V_k} < \varepsilon.
      \]
      We have the following claim.

      \textbf{Claim:} For all $n > n_0$, for all $y \in A$, we have
      \[
      d(\ms h_n (y), \ms h_\infty (y)) < \varepsilon.
      \]
      \textbf{Proof of Claim:} Fix an $n > n_0$ and let $y \in A$ be
      arbitrarily chosen. We have two subcases.
      \begin{enumerate}
        \item First, suppose $\ms h_{n_0}(y) \not\in
          \bigcup_{k>n_0}^\infty V_k$.

          Recall that we defined the $h_k$ such that each $h_k$ is
          identity outside $V_k$. It follows that for all $k > n_0$ we
          have $h_k(y) = y$. Hence
          \begin{align*}
            \ms h_n (y)
            &= \ms h_\infty (y),
          \end{align*}
          so
          \begin{align*}
            d(\ms h_n(y), \ms h_{\infty}(y))
            &= 0 \\
            &< \varepsilon,
          \end{align*}
          as desired. \cmark
          \begin{figure}[H]
            \centering
            \includegraphics[width=.5\linewidth]{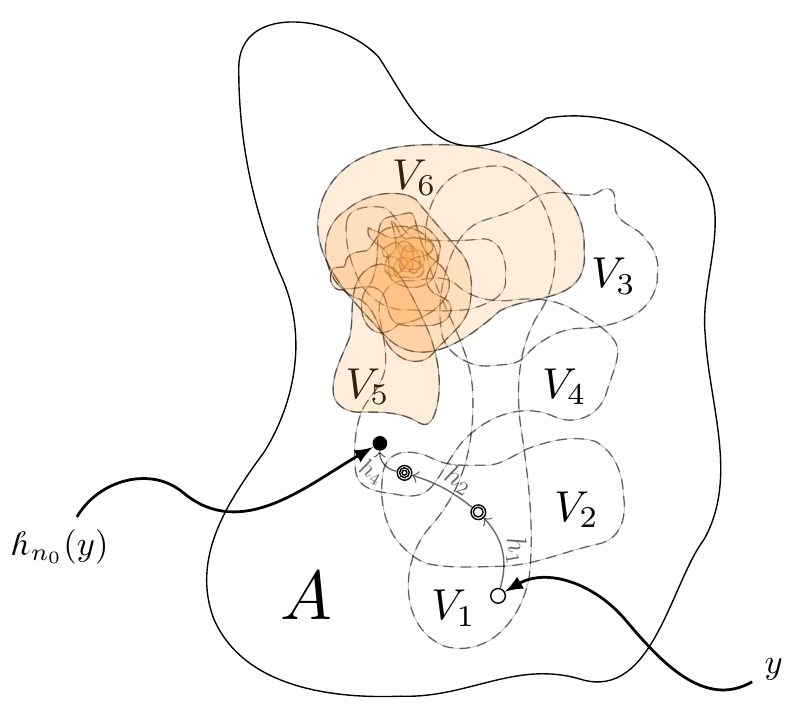}
            \caption{An example of this case with $n_0 = 4$. The
              shaded portions represent $\bigcup_{k>n_0}^\infty V_k$.
              Here, $y$ starts in $V_1$, is mapped into $V_2$ by
              $h_1$, into $V_4$ by $h_2$, skipped by $h_3$, then
              finally mapped to another point of $V_4$ by $h_4$,
              before remaining fixed for all $k > 4$. }
          \end{figure}
        \item Now, suppose $\ms h_{n_0} (y) \in
          \bigcup_{k>n_0}^\infty V_k$. Note, since \np{for all $k >
          n_0$, $h_k$ is bijective and is identity outside
          $\bigcup_{k>n_0}^\infty V_k$}, it follows that for all $n >
          n_0$,
          \[
          \ms h_{n}(y) \in\ \ \bigcup_{\mathclap{k>n_0}}^\infty\,
          V_k
          \]
          and hence
          \[
          \ms h_{\infty} (y) \in \ol{\ \
          \bigcup_{\mathclap{k>n_0}}^\infty\, V_k}.
          \]
          Note, the set in the second is just the closure of the set
          in the first; thus they have the same diameter. By
          definition of $n_0$, we have $\diam
          \pn{\bigcup_{k>n_0}^\infty V_k} < \varepsilon$, hence
          \begin{align*}
            d(\ms h_n(y), \ms h_\infty (y)) < \varepsilon
          \end{align*}
          as desired. \cmark
      \end{enumerate}
      In either case, we get that $d(\ms h_n(y),\ \ms h_\infty (y)) <
      \varepsilon$. Now (writing the restrictions explicitly again) it
      follows that $\ms h_n|_{A}$ is a sequence of homeomorphisms with
      $\pn{\ms h_n|_{A}} \uconv \pn{\ms h_\infty|_{A}}$.

      Finally, recall that by hypothesis, $\ms h_\infty$ is a
      bijection. This implies $\ms h_\infty|_{A}$ is too, and since
      $A$ is compact, \cref{cor:uniformly-convergent-homeomorphism}
      now guarantees $\ms h_\infty|_{A}$ is a homeomorphism on $A$.
  \end{enumerate}
  Now, applying the gluing lemma to $(\ms h_\infty)|_{A}$ and $(\ms
  h_\infty)|_{A^c}$ we conclude that $\ms h_\infty$ is continuous.
  An identical argument shows $\ms h_\infty^{-1}$ is continuous. It
  follows that $\ms h_\infty$ is a homeomorphism, as desired.
\end{proof}

\subsection{Iteratively Constructing (Ambient) Isotopies}
We now state the analogous result for isotopies. As with
\cref{thm:vks-ambient-homeomorphism}, the theorem below is valid for
isotopies in general but we are only interested in applying it to
ambient isotopies. It might be easy to get bogged down by the
additional details so we summarize the main ideas.

Given a sequence of isotopies $H_k$, if the associated homeomorphisms
$h_k(\cdot) \coloneqq H_k(1, \cdot)$ satisfy the hypotheses of
\cref{thm:vks-ambient-homeomorphism}, then we can stitch the $H_k$'s
together into an isotopy $\ms H_\infty$ as follows. First, define $t_0
= 0$ and let $\pn{t_k}_{k=1}^\infty$ be a strictly increasing sequence
in $\pn{0,1}$. Define $\ms H_\infty$ to apply the effects of $H_1$
over the compressed time interval $\bk{t_0, t_1}$. Then, do the same
to apply $H_2$ over $\bk{t_1, t_2}$. Continue this process, in general
applying $H_k$ over the interval $\bk{t_{k-1}, t_k}$.

Stopping the construction above after $n$ steps will give us an
isotopy $\ms H_n$. Taking $n \to \infty$ we will get a function $\ms
H_\infty$ with $\ms H_\infty(1, \cdot) = \ms h_\infty$. By
\cref{thm:vks-ambient-homeomorphism}, $\ms h_\infty$ will be a
homeomorphism. And since the $\ms H_n$ are all isotopies, we'll see
that $\ms H_\infty(t, \cdot)$ will be a homeomorphism for all $t \in
\bp{0,1}$. Applying a uniform convergence argument to the $\ms H_n$
will then show $\ms H_\infty$ is continuous overall and thus an
isotopy!
\begin{theorem}\label{thm:vks-ambient-isotopy}
  Let $\metspace{Y}$ be a metric space. For all $k \in \NN$, let $H_k
  \colon [0,1] \times Y \to Y$ be an isotopy, and let $V_k \subseteq
  Y$ such that $H_k$ is identity on $[0,1] \times (V_k^c)$. For each
  $k$ define $h_k \colon Y \to Y$ by $h_k(y) = H_k(1, y)$; note that
  by definition of an isotopy, $h_k$ is a homeomorphism.

  Suppose that the $h_k$'s and $V_k$'s satisfy the hypotheses of
  \cref{thm:vks-ambient-homeomorphism}, and for all $n \in \NN$ define
  $\ms h_n = \comp_{k=1}^n h_k$. Let $t_0 = 0$ and let
  $\pn{t_k}_{k=1}^\infty$ be a strictly increasing sequence in
  $\pn{0,1}$ converging to $1$. Then $\ms H_\infty \colon [0,1] \times
  Y \to Y$ defined by
  \[
    \ms H_\infty(t,y) =
    \begin{cases}
      H_1\pn{\frac{t - t_0}{t_1 - t_0},\ y} & \text{ if } t
      \in \bk{t_0, t_1} \\
      H_2\pn{\frac{t - t_1}{t_2 - t_1},\ \ms h_1(y)} & \text{ if } t \in
      \pb{t_1, t_2} \\
      H_3\pn{\frac{t - t_2}{t_3 - t_2},\ \ms h_2(y)} & \text{ if } t \in
      \pb{t_2, t_3} \\
      \,\,\,\vdots & \\
      H_{k}\pn{\frac{t - t_{k-1}}{t_{k} - t_{k-1}},\ \ms h_{k-1}(y)} & \text{ if } t \in
      \pb{t_{k-1}, t_{k}} \\
      \,\,\,\vdots &  \\
      \ms h_\infty(y) & \text{ if } t = 1,
    \end{cases}
  \]
  is an isotopy.
\end{theorem}
\begin{proof}
  To show $\ms H_\infty$ is an isotopy we must show
  \begin{enumerate}
    \item $\ms H_\infty(0, \cdot)$ is
      identity, \label{cond:is-identity-at-0}
    \item For each $t \in [0,1]$, $\ms H_\infty(t, \cdot)$ is a
      homeomorphism, and \label{cond:is-homeomorphism-at-each-t}
    \item $\ms H_\infty$ is
      continuous. \label{cond:is-continuous-overall}
  \end{enumerate}
  We prove these in the order above.
  \begin{enumerate}
    \item For all $y \in Y$, $\ms H_\infty(0, \cdot) = H_1(0, \cdot)$.
      Since $H_1$ is an isotopy, $H_1(0, \cdot)$ is identity (by
      definition) and this proves (\ref{cond:is-identity-at-0}).
    \item To prove (\ref{cond:is-homeomorphism-at-each-t}) we break
      things up into three subcases.
      \begin{enumerate}
        \item Suppose $t = 0$. Then $\ms H_\infty(t, \cdot) = \ms
          H_\infty(0, \cdot)$ which is the identity and thus a
          homeomorphism. \cmark
        \item Suppose $t \in \pn{0,1}$. Then there exists $k \in \NN$
          such that $t \in \pb{t_{k-1}, t_{k}}$. Recall that by
          construction,
          \begin{equation}
            \ms H_\infty(t,\cdot) = H_{k}\pn{\frac{t - t_{k-1}}{t_{k} -
                t_{k-1}},\ \ms h_{k-1}(\cdot)}. \label{eq:h-inf-h-k}
          \end{equation}
          Define $\ms g(\cdot) \coloneqq H_k\pn{\frac{t - t_{k-1}}{t_k
          - t_{k-1}},\ \cdot}$. By definition of an isotopy, $\ms g$
          is a homeomorphism. Also, $\ms h_{k-1}$ is a finite
          composition of homeomorphisms and thus a homeomorphism.
          Rewriting \cref{eq:h-inf-h-k} gives us that
          \[
          \ms H_\infty(t, \cdot) = \ms
          (g \circ \ms h_{k-1})(\cdot)
          \]
          which is a finite composition of homeomorphisms and hence
          itself a homeomorphism, as desired. \cmark
        \item Now suppose $t = 1$. Then $\ms H_\infty(t, \cdot) = \ms
          h_\infty(\cdot)$. Since the $h_k$'s, $V_k$'s were assumed
          to satisfy the hypotheses of
          \cref{thm:vks-ambient-homeomorphism} we get that $\ms
          h_\infty$ is a homeomorphism as desired. \cmark
      \end{enumerate}
      In either case, we see $\ms H_\infty(t, \cdot)$ is a
      homeomorphism.
    \item Finally, it remains to show $\ms H_\infty$ is continuous. We
      employ uniform convergence.

      Define a sequence of isotopies $\ms H_n$ as follows: For each $n
      \in \NN$, let $\ms H_n \colon [0,1] \times Y \to Y$ be given by
      \[
      \ms H_n(t,y) =
      \begin{cases}
        H_1\pn{\frac{t - t_0}{t_1 - t_0},\ y} & \text{ if } t
        \in \bk{t_0, t_1} \\
        H_2\pn{\frac{t - t_1}{t_2 - t_1},\ \ms h_1(y)} & \text{ if } t \in
        \pb{t_1, t_2} \\
        \,\,\,\vdots & \\
        H_{n}\pn{\frac{t - t_{n-1}}{t_{n} - t_{n-1}},\ \ms h_{n-1}(y)}
        & \text{ if } t \in \pb{t_{n-1}, t_{n}} \\
        \ms h_n(y) & \text{ if } t \in \pb{t_{n}, 1}.
      \end{cases}
      \]
      That is, we follow $\ms H_\infty(t, y)$ until we reach
      $t=t_{n}$ and then we freeze. One can verify that the $\ms
      H_n(t, y)$ are indeed isotopies; of particular note, they are
      continuous. We now show $\ms H_n \uconv \ms H_\infty$.

      Let $\varepsilon > 0$ be given. Then by the hypotheses on the
      $V_k$ there exists $n_0 \in \NN$ such that
      \[
      \diam\pn[bigg]{\bigcup_{k>n_0}^\infty V_k} < \varepsilon.
      \]
      Let $n > n_0$ be arbitrarily chosen, and similarly let $(t, y)
      \in [0,1] \times Y$. We show $d(\ms H_n(t,y), \ms H_\infty(t,
      y)) < \varepsilon$. We have two subcases.
      \begin{enumerate}
        \item First, suppose $t \in \bk{0, t_{n}}$. Then
          $\ms H_n(t, y) = \ms H_\infty(t,y)$ and so we have $d(\ms
          H_n(t,y), \ms H_\infty(t,y)) = 0$ and the bound holds.
        \item Now, suppose $t \in \pb{t_{n}, 1}$. If $y
          \in \pn{\bigcup_{k>n_0}^\infty V_k}^c$, then $\ms
          H_n(t, y) = \ms H_\infty(t, y)$, so we have $d(\ms H_n(t,y),
          \ms H_\infty(t,y)) = 0$ and the bound holds. Else, note that
          both of $\ms H_n(t,y)$, $\ms H_\infty(t,y) \in
          \ol{\bigcup_{k>n_0}^\infty V_k}$, hence
          \[
          d(\ms H_n(t,y), \ms H_\infty(t,y)) < \varepsilon
          \]
          as desired.
      \end{enumerate}
      In either case, we have $d(\ms H_n(t,y), \ms H_\infty(t,y)) <
      \varepsilon$. As $(t,y)$ were arbitrarily chosen, this implies
      $\ms H_n \uconv \ms H_\infty$. By
      \cref{prop:uniform-convergence-continuity}, $\ms H_\infty$ is
      continuous.
  \end{enumerate}
  It follows that $\ms H_\infty$ is an isotopy as desired.
\end{proof}
In the next section, we apply this result to a variety of curves,
beginning with the example from \cref{fig:scintillating-curves}.

\section{Various Applications of
  \cref{thm:vks-ambient-isotopy}}\label{sec:various-applications} The
first few examples will all make use of the following lemma, which
allows us to remove the bijectivity hypothesis from
\cref{thm:vks-ambient-isotopy}.
\begin{lemma}\label{lem:disjoint-vks}
  Let all variables be quantified as in
  \cref{thm:vks-ambient-isotopy}. Then if the $V_k$'s are all
  disjoint, $\ms H_\infty(1, \cdot)$ is guaranteed to be a bijection.
\end{lemma}
\begin{proof}
  If the $V_k$ are all disjoint, then defining $U =
  \bigcup_{k=1}^\infty V_k$ we can write $Y$ as the disjoint union
  \[
    Y = (U^c) \sqcup \pn[bigg]{\bigsqcup_{k=1}^\infty V_k}.
  \]
  Note, $\ms H_\infty(1, \cdot)$ is identity on $U^c$ and hence a
  bijection. Since the $V_k$ are all disjoint, $\ms H_\infty(1,
  \cdot)|_{V_k} = \ms H_k(1, \cdot)|_{V_k}$, the latter of which is a
  homeomorphism and thus bijective, so $\ms H_\infty(1, \cdot)$ is a
  bijection on each of the $V_k$.

  Thus $\ms H_\infty(1, \cdot)$ is a bijection overall.
\end{proof}

\begin{proposition}\label{prop:wild-looking-unknot-tame}
  The curve shown in \cref{fig:wild-looking-unknot-redux} below is
  tame.\footnote{Strictly speaking we have not defined tameness for
    curves, only for knots. A \emph{tame curve} is a curve that's
    ambient homeomorphic (equivalently, ambient isotopic) to a
    polygonal curve.} In particular, it is an unknot.
  \begin{figure}[H]
    \centering
    \centering
    \includegraphics[width=.4\linewidth]{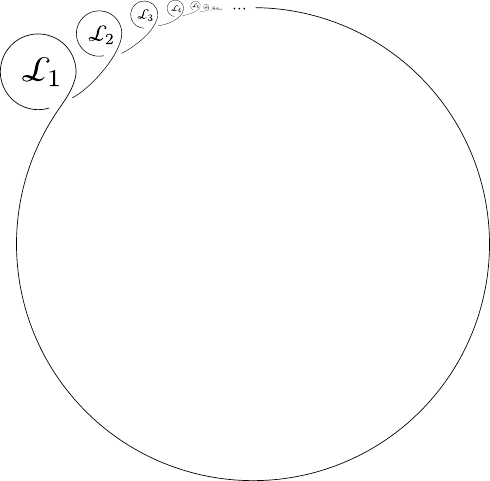}
    \caption{A wild-looking unknot, redux.}
    \label{fig:wild-looking-unknot-redux}
  \end{figure}
\end{proposition}
\begin{proof}
  Let $f_0 \colon S^1 \into \RR^3$ be the standard unknot, and let
  $f_1 : S^1 \into \RR^3$ be an embedding yielding a diagram like
  \cref{fig:wild-looking-unknot-redux}.\footnote{A parametrization can
    be found in \cite{KobayashiThesis}, although it is given in the
    context of a ``theorem'' about tameness and parametrizations that
    turns out to be incorrect.} We apply
  \cref{thm:vks-ambient-isotopy} to construct an ambient isotopy $\ms
  H_\infty \colon [0,1] \times \RR^3 \to \RR^3$ taking $f_0$ to $f_1$.

  Consider the sequence of $V_k$'s chosen as follows.
  \begin{figure}[H]
    \centering
    \includegraphics[scale=.508]{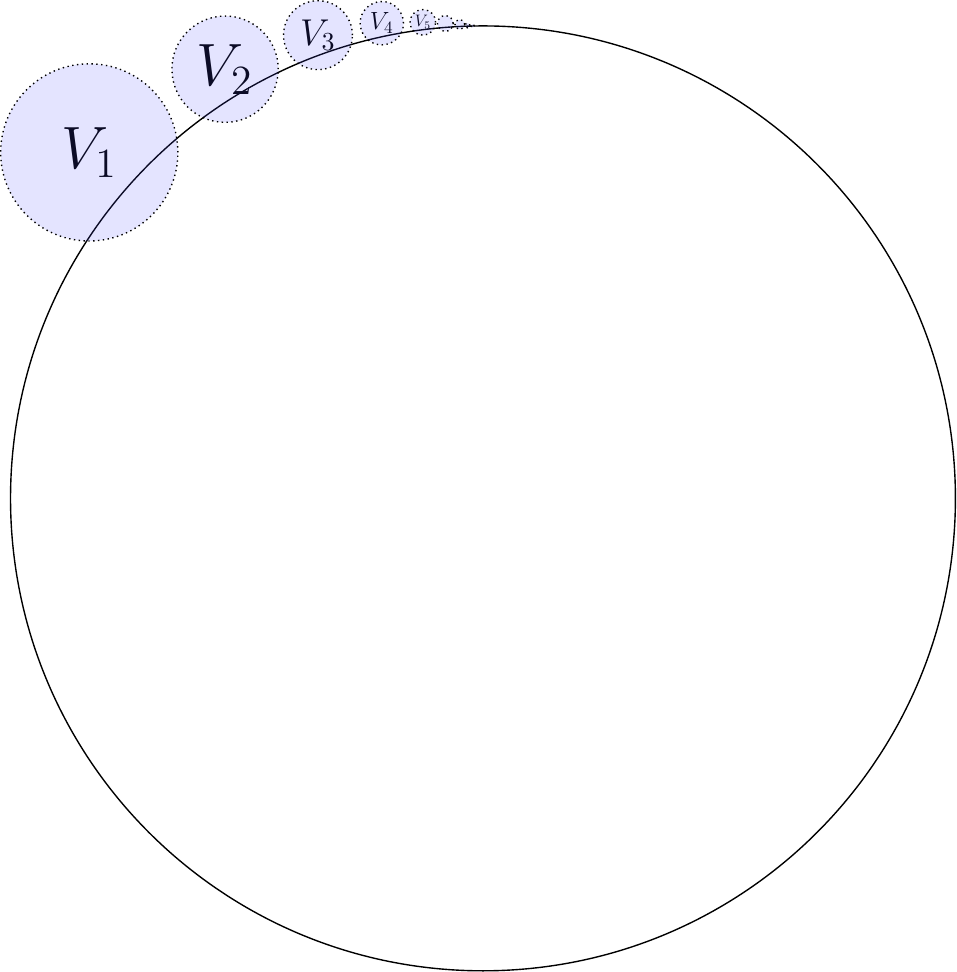}
    \caption{$V_k$'s.}
  \end{figure}
  One can verify that $\lim_{n\to\infty} \diam\pn{\bigcup_{k=n}^\infty
    V_k} = 0$ and that there exists a compact set $A \subseteq \RR^3$
  such that $\bigcup_{k=1}^\infty V_k \subseteq A$.

  For all $k \in \NN$, let $H_k \colon [0,1] \times \RR^3 \to \RR^3$
  be an ambient isotopy inserting a Reidemeister I into the arc bound
  in $V_k$. Use these $H_k$ to define $\ms H_\infty$ as in
  \cref{thm:vks-ambient-isotopy}. By \cref{lem:disjoint-vks}, $\ms
  H_\infty(1, \cdot)$ is a bijection. Thus by
  \cref{thm:vks-ambient-isotopy}, $\ms H_\infty$ is an ambient isotopy
  from $f_0$ to $f_1$.
\end{proof}

\begin{proposition}\label{prop:countable-r2}
  The following curve is tame.
  \begin{figure}[H]
    \centering
    \includegraphics[angle=-90]{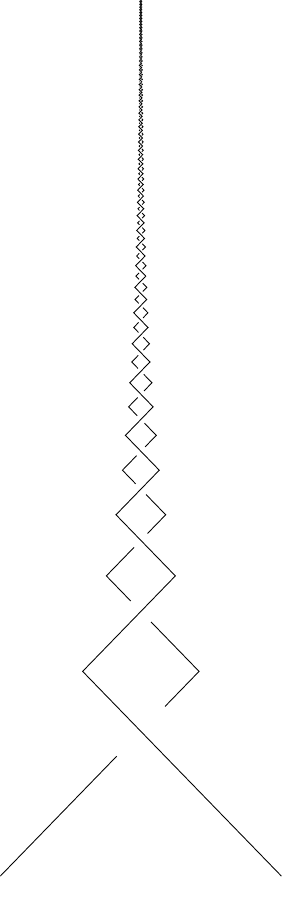}
    % \caption{An arc created by applying a countable sequence of
    % Reidemeister II moves.
  \end{figure}
\end{proposition}
\begin{proof}[Proof (Sketch)]
  We apply \cref{thm:vks-ambient-isotopy} twice. This two-step method
  is not strictly necessary, but the diagram is a bit less cluttered
  this way. Consider the following starting curve:
  \begin{figure}[H]
    \centering
    \includegraphics[angle=-90]{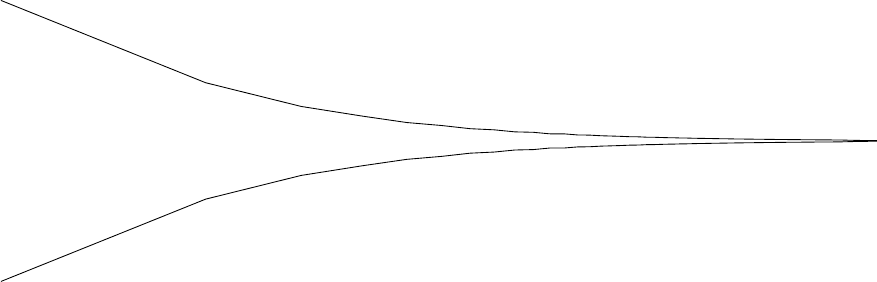}
    % \caption{A countable sequence of Reidemeister II moves.}
    % \label{fig:countable-r2s}
  \end{figure}
  Apply Reidemeister II moves within the dotted regions below:
  \begin{figure}[H]
    \centering
    \includegraphics{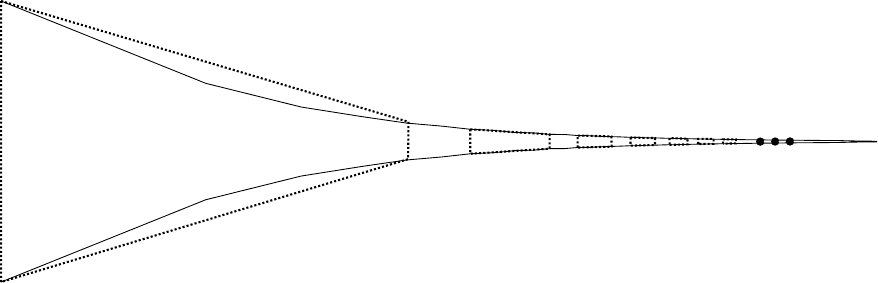}
  \end{figure}
  Since these $V_k$ are disjoint, we can again apply
  \cref{lem:disjoint-vks} to obtain an ambient isotopy. The result
  looks something like the following:
  \begin{figure}[H]
    \centering
    \includegraphics[angle=-90]{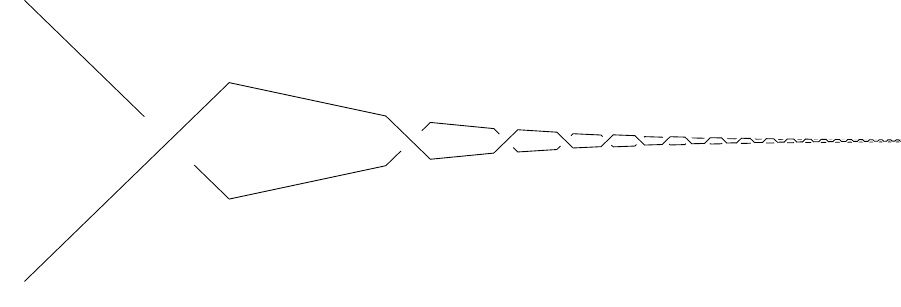}
  \end{figure}
  Now, perform Reidemeister II moves in the following regions:
  \begin{figure}[H]
    \centering
    \includegraphics[angle=-90]{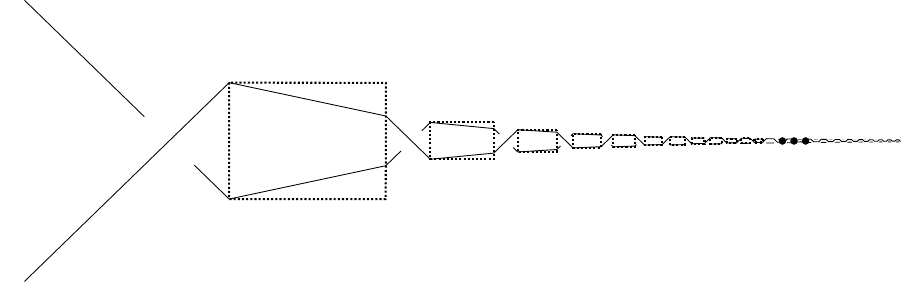}
  \end{figure}
  Again, the $V_k$ here are all disjoint, hence one can apply
  \cref{lem:disjoint-vks} to show that this is an ambient isotopy. The
  end result is
  \begin{figure}[H]
    \centering
    \includegraphics[angle=-90]{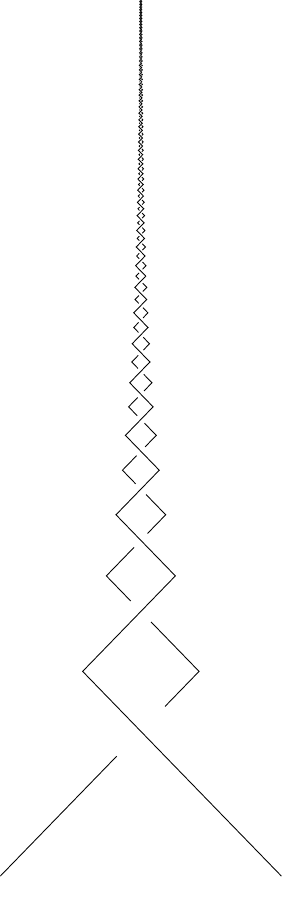}
  \end{figure}
  which is the desired diagram.
\end{proof}

We now consider a similar curve, this time constructed using
Reidemeister I moves. This will be the most technical argument of the
paper. We advise the reader to read through
\cref{ex:remarkable-curve-redux} in the next section before
continuing. This is because \cref{ex:remarkable-curve-redux} shows how
we can lose bijectivity in the limit, and the bulk of the challenge in
\cref{prop:countable-r1-2} is addressing similar concerns. We have to
address bijectivity in a manner like this whenever we have points that
are moved by infinitely many of the $H_k$'s (whereas in
\cref{prop:countable-r2}, each point is moved by only finitely many
$H_k$).

\begin{proposition}\label{prop:countable-r1-2}
  Let $f_0 \colon [0,1] \into \RR^3$ be
  \begin{figure}[H]
    \centering
    \includegraphics[scale=.625]{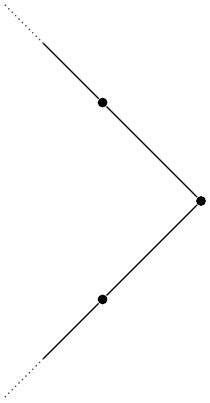}
    \caption{Our starting arc.}
    \label{fig:recursive-r1-starting-arc}
  \end{figure}
  and let $f_1 \colon [0,1] \into \RR^3$ be
  \begin{figure}[H]
    \centering
    \includegraphics[angle=-90]{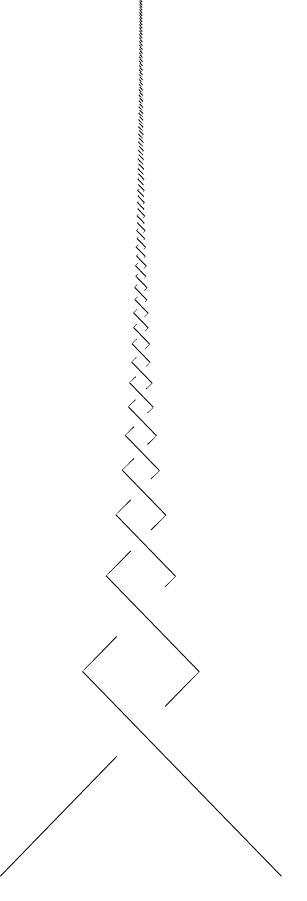}
    \caption{A different countable sequence of Reidemeister I moves.}
    \label{fig:another-countable-r1}
  \end{figure}
  Then $f_0 \cong f_1$.
\end{proposition}
\begin{proof}[Proof (Sketch)]
  We will construct the ambient isotopy from $f_0$ to $f_1$ by a
  recursive process. We will repeatedly insert Reidemeister I moves
  like the following:
  \begin{figure}[H]
    \centering
    \includegraphics{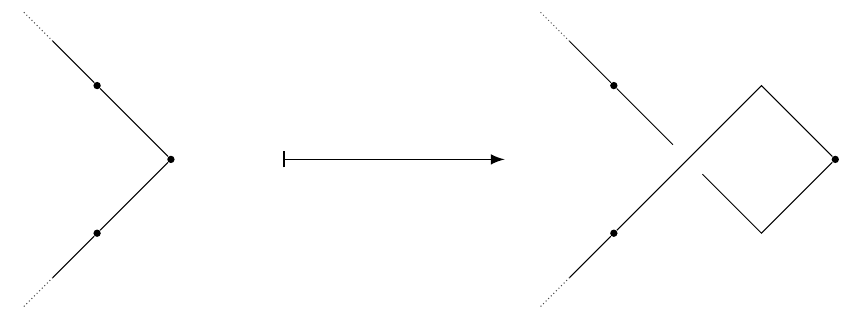}
    \caption{The general procedure.}
    \label{fig:general-procedure}
  \end{figure}
  We must do so in such a way that we can still argue bijectivity of
  $\ms H_\infty(1, \cdot)$. The key idea is to choose our $H_k$'s so
  that only one point (denoted $y_\infty$) gets moved infinitely-many
  times.\footnote{$y_\infty$ will be the point that gets sent to the
    limit of the twists in \cref{fig:another-countable-r1}. In our
    construction, $y_\infty$ will be the vertex in
    \cref{fig:recursive-r1-starting-arc}, but one can create other
    constructions where it is a different point.} We explicitly
  guarantee this by constructing our $H_k$'s so that for all $y \in
  \RR^3 \setminus \set{y_\infty}$, there exists $n_0$ such that for $n
  > n_0$, $\ms h_{n_0}(y)$ is unmoved by $H_n(t, \cdot)$.

  To that end, define $\ell$ as shown in \cref{fig:with-l}, and let
  $\varepsilon > 0$ with $\varepsilon \ll 1$.\footnote{Actually,
    $\varepsilon > 0$ can be arbitrarily chosen so long as for all $k
    \in \NN$ we have $\ol{V_{k+1}} \subseteq V_{k}^\circ$. We just
    choose $\varepsilon \ll 1$ because it makes for cleaner-looking
    diagrams.} Define $V_1$ to be a closed rectangular prism of
  dimensions $(6 + \varepsilon)\ell \times (2+\varepsilon)\ell \times
  (2+\varepsilon)\ell$, and $\msf{H_1}$ to be a PL ambient isotopy
  inserting the first loop such that $\msf{H}_1$ is identity off
  $V_1$.\footnote{We can assume PL-ness because the modifications can
    be realized by \emph{elementary
      moves}.}\textsuperscript{,}\footnote{The $6$ in our prism
    dimensions comes from the fact $\ell$ is defined to be
    $1/3$\textsuperscript{rd} of the length of the twist inserted in
    \cref{fig:with-l}, and the moves halve in size at each iteration.}
  Note, even though we define $V_1$ to be closed, we'll draw it with
  dotted lines in the below.
  \begin{figure}[H]
    \centering
    \includegraphics{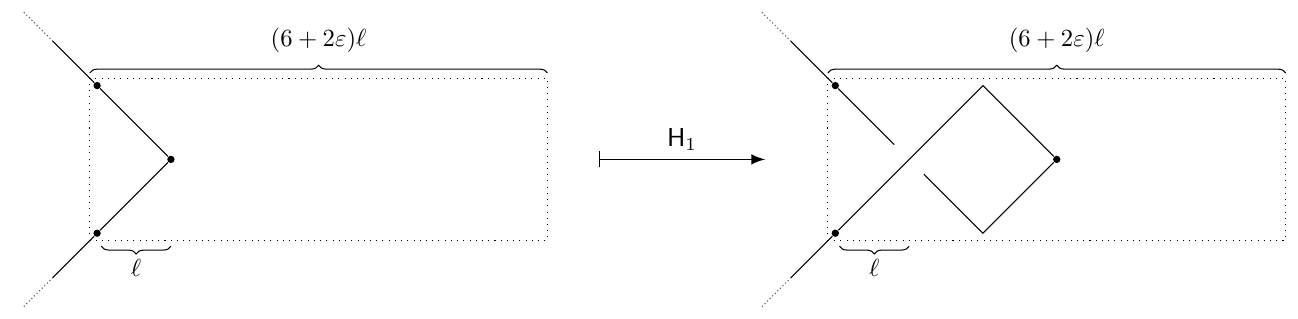}
    \caption{The same figure, now showing $V_1$.}
    \label{fig:with-l}
  \end{figure}

  Now, we describe the general strategy for inserting the
  $k+1$\textsuperscript{st} loop given the first $k$ loops. We want
  $V_{k+1}$, $\msf{H}_{k+1}$ to be half-scale versions of $V_k$,
  $\msf{H}_k$. The figure below shows this for $k = 1$.
  \begin{figure}[H]
    \centering
    \includegraphics{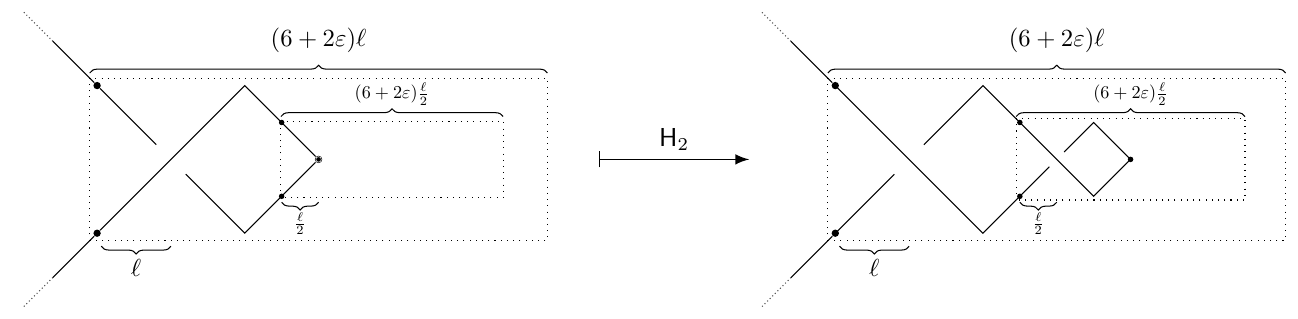}
    \caption{$V_1$ and $V_2$, with application of $\msf H_2$ shown.}
    \label{fig:v2}
  \end{figure}
  But, to make things work, it will be important to first apply some
  \emph{intermediate} ambient isotopy $\msf{H}_{k+.5}$ in between
  $\msf{H}_k$ and $\msf{H}_{k+1}$ such that $\msf H_{k+.5}$ does not
  change the diagram, but {does} help ensure that points in the
  ambient space aren't lost in the limit.

  \textbf{Desired Properties of $\bm{\msf H_{k+.5}}$:} We want $\msf
  H_{k+.5}$ to preemptively ``unsquish'' points that might be
  compressed together by $\msf H_{k+1}$. To determine exactly how much
  unsquishing we have to do, we look at a sort of inverse Lipschitz
  condition.

  Let $\msf{h}_{k+1} = \msf{H}_{k+1}(1, \cdot)$. Since $\msf{H}_{k+1}$
  is a PL ambient isotopy, $\msf{h}_{k+1}$ is a PL ambient
  homeomorphism, and thus there exists $c \in (0,1)$ such that for all
  $x_1, x_2 \in \RR^3$,
  \begin{equation}
    c \cdot d(x_1,x_2) \leq d(\msf{h_{k+1}}(x_1),
    \msf{h_{k+1}}(x_2)).\footnote{This essentially pops out of the
      finiteness condition on our simplicial complexes for PL
      maps.} \label{eq:bound-on-squishing}
  \end{equation}
  Note that because the $\msf H_{k}$ are all identical up to scaling,
  $c$ is independent of $k$.\footnote{One might ask why we can't have
    $c \geq 1$. Note that $V_{k+1}$ being bounded precludes $c > 1$.
    For $c = 1$, note that $\msf{h}_{k+1}$ is not a vector space
    isomorphism of $\RR^3$, and hence not an isometry on $\RR^3$;
    since $\msf h_{k+1}$ is identity outside $V_{k+1}$, isometry must
    fail on $V_{k+1}$.} We interpret \cref{eq:bound-on-squishing} as
  giving us a bound on how much $\msf{h}_{k+1}$ can ``squish'' points
  in the space together. Let $q_k$ be the tip of the twist before
  applying $\msf H_{k+1}$:
  \begin{figure}[H]
    \centering
    \includegraphics[scale=.75]{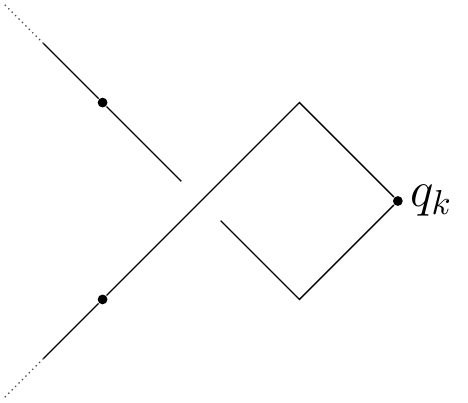}
    \caption{$q_k$ labeled.}
  \end{figure}
  For all $p \in V_{k+1}$, we want $\msf H_{k+.5}$ to be constructed
  to guarantee that either
  \begin{enumerate}
    \item $\msf{h}_{k+1}(\msf h_{k+.5}(p)) \in V_k \setminus V_{k+1}$
      (i.e.\ $p$ gets moved to the outer box), or
    \item $p$ gets ``moved farther from $q_k$ than it can be squished
      in later'':
      \begin{equation}
        \frac{1}{c} \cdot d(p, q_k) \leq d(\msf h_{k +.5}(p),
        \msf h_{k+.5}(q_k)). \label{eq:h1.5-bound}
      \end{equation}
  \end{enumerate}

  \textbf{Constructing $\bm{\msf H_{k+.5}}$:} Let $V_{k+.5}$ be a
  slightly-scaled-up version of $V_{k+1}$ such that $V_{k+1}
  \subsetneq V_{k+.5} \subsetneq V_k$. To make things easier, we will
  require that $V_{k+.5}$ also only intersects with $(\msf{h}_{k}
  \circ \msf h_{k-1} \circ \cdots \circ \msf h_{1} \circ f_0)([0,1])$
  in a wedge-shape and that $V_{k+.5}$ and $V_{k+1}$ share the same
  center of mass and have all sides parallel (see
  \cref{fig:v1.5-and-v2}).
  \begin{figure}[H]
    \centering
    \includegraphics[scale=.75]{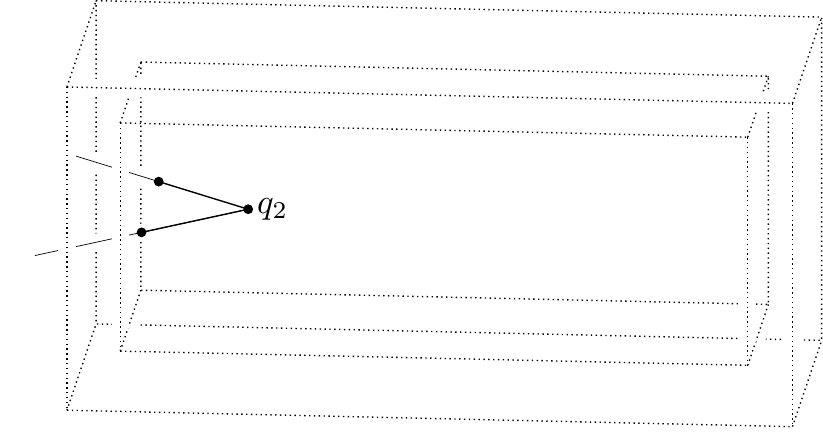}
    \caption{$V_{k+.5}$ and $V_{k+1}$}
    \label{fig:v1.5-and-v2}
  \end{figure}
  We can parameterize every point $p \in V_{k+.5}$ in terms of a
  piecewise linear function as detailed below. The construction is a
  bit unergonomic to formalize explicitly, but it is meant to capture
  the ideas of \cref{fig:vs-with-lines}.
  \begin{figure}[H]
    \centering
    \includegraphics[scale=1.25]{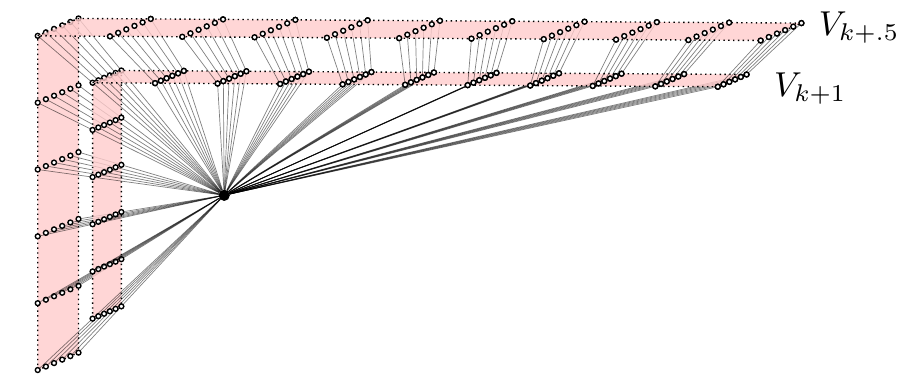}
    \caption{The piecewise-linear parameterization, with some of the
      $v_{k+.5}$'s (outer points; defined below) and $v_{k+1}$'s
      (inner points; also defined below) shown in white.}
    \label{fig:vs-with-lines}
  \end{figure}
  We construct it in two parts and then scale them by half and glue
  them together.
  \begin{enumerate}
    \item If $p \in V_{k+.5} \setminus V_{k+1}$, there exist unique
      points $v_{k+.5} \in \partial {V_{k+.5}}$ and $v_{k+1} \in
      \partial {V_{k+1}}$ such that $v_{k+.5}$ is the point in
      $V_{k+.5}$ ``corresponding'' to $v_{k+1}$ in $V_{k+1}$, and $p$
      is on the line segment $\ol{v_{k+.5} v_{k+1}}$.\footnote{By
      ``corresponds,'' we mean that given a linear function that
      scales up $V_{k+1}$ to yield $V_{k+.5}$, $v_{k+1}$ gets mapped
      to $v_{k+.5}$.} Thus there exists a unique $s \in [0,1]$ such
      that we can write $p$ as a convex combination
      \[
      p = s \cdot v_{k+.5} + (1-s) \cdot v_{k+1}.
      \]
    \item If $p \in V_{k+1}$, there exists a unique point $v_{k+1}
      \in\partial {V_{k+1}}$ such that $p$ is on the line segment
      $\ol{q_k v_{k+1}}$. Analogously to the above, there exists a
      unique $s \in [0,1]$ such that we can write $p$ as
      \[
      p = s \cdot v_{k+1} + (1-s) \cdot q_k.
      \]
  \end{enumerate}
  We re-scale $s$ to glue these two parameterizations into a single
  function which we call $\msf H'_{k+.5}$:
  \[
    \msf H'_{k+.5}(s, v_{k+.5}, v_{k+1}) =
    \begin{cases}
      \qquad \ \, 2s \cdot v_{k+1} \, + (1-2s) \cdot q_k & s \in
      \bk{0, \frac{1}{2}} \\
      (2s-1) \cdot v_{k+.5} + (2-2s) \cdot v_{k+1} & s \in
      \pb{\frac{1}{2}, 1}.
    \end{cases}
  \]
  We'll now do something a bit bizarre and rewrite the parameters in
  $\msf H'_{k+.5}$ as functions of $p$. Note that with the re-scaling
  of $s$, we now have $s$ uniquely determined by $p$. Also recall that
  by construction, $v_{k+.5}$ and $v_{k+1}$ are each uniquely
  determined by $p$. Hence, we can think of $s$, $v_{k+.5}$, $v_{k+1}$
  as being functions of $p$. One can show that these are all
  continuous. As such, we can indeed think of $\msf H'_{k+.5}$ as just
  being a complicated way of writing the identity function on
  $V_{k+.5}$.

  To turn $\msf H'_{k+.5}$ into our ambient isotopy $\msf H_{k+.5}$,
  we now introduce time dependence in $s$. Define $s_{c_0} =
  \frac{c}{2}$ % \footnote{The $\frac{1}{2}$ factor pops out of the fact
  % that $\ms V_{1.5}(s, p)$ changes behavior at $s = \frac{1}{2}$.}
  and $s_{c_1} = \frac{1}{2}$, and observe $s_{c_0} <
  s_{c_1}$.\footnote{The reason that $c$ appears in this expression is
    because we're trying to get \cref{eq:h1.5-bound} out in the end.}
  Define $s_c \colon [0,1] \to [s_{c_0}, s_{c_1}]$ by
  \[
    s_c(t) = t \cdot s_{c_1} + (1-t) \cdot s_{c_0},
  \]
  and use this to define
  \[
    s'(t, p) =
    \begin{cases}
      \pn{\frac{s(p)}{s_{c_0}}} \cdot s_c(t) & \text{if } s(p) \in
      \bk{0, s_{c_0}} \\
      \pn{\frac{s(p) - s_{c_0}}{1 - s_{c_0}}} \cdot 1 + \pn{1 -
        \pn{\frac{s(p) - s_{c_0}}{1 - s_{c_0}}}} \cdot s_c(t) &
      \text{if } s(p) \in \pb{s_{c_0}, 1}
    \end{cases}
  \]
  This looks unpleasant but the idea is simple. First, recall that
  $s(p)$ represents a parameter in $[0,1]$ that tells us how to write
  $p$ as a convex combination of other points. One can verify that
  when $t = 0$, $s'(t,p)$ reduces to $s(p)$. Then, as $t$ increases to
  $1$, $s'(t,p)$ distorts the interval represented by $s(p)$ until we
  end up with something like the following, in which $s_{c_0}$ gets
  mapped to where $s_{c_1}$ was initially:
  \begin{figure}[H]
    \centering
    \includegraphics[scale=1.5]{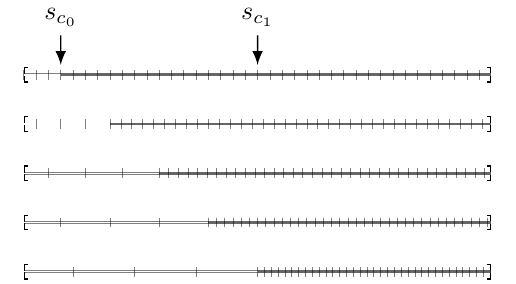}
    \caption{The interval $[0,1]$ represented by $s'(t,p)$ as $t$ goes
      from $0$ to $1$. The light portion represents the values where
      $s(p) \in [0, s_{c_0}]$ and the dark portion represents $s(p)
      \in [s_{c_0}, 1]$.}
  \end{figure}
  The net effect of $\msf H_{k+.5}$ is to take a diagram like the
  following
  \begin{figure}[H]
    \centering
    \includegraphics{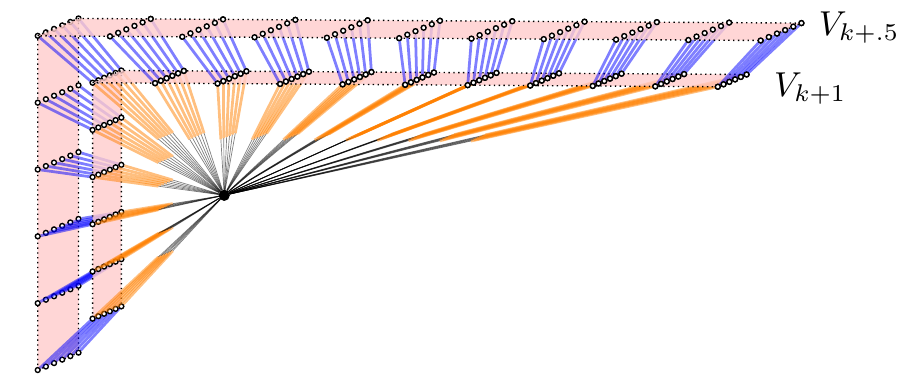}
  \end{figure}
  and turn it into
  \begin{figure}[H]
    \centering
    \includegraphics{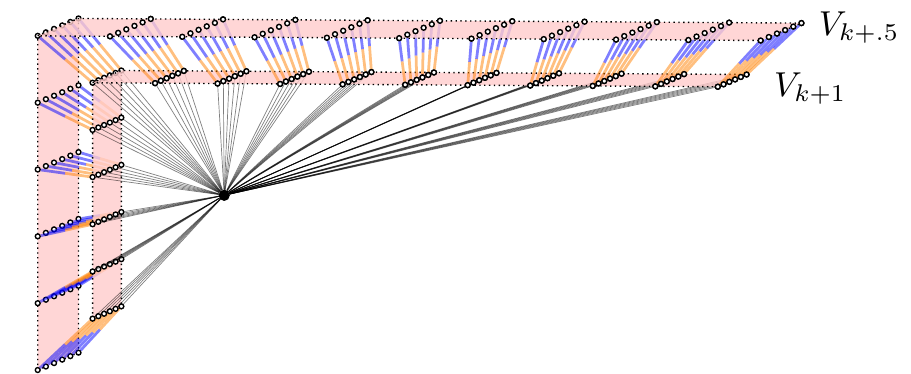}
    \caption{The effects of $\msf H_{k+.5}$.}
  \end{figure}
  Finally, we have the following:
  \begin{adjustwidth}{1em}{}\vspace{.5em}
    \begin{leftbar}
      \textbf{Claim:} $\msf H_{k+.5} \colon [0,1] \times \RR^3 \to \RR^3$
      given by
      \[
        \msf H_{k+.5}(t, p) =
        \begin{cases}
          \msf H'_{k+.5}(s'(t, p), p) & p \in V_{k+.5} \\
          p & p \not \in V_{k+.5}
        \end{cases}
      \]
      satisfies the desired properties of $\msf H_{k+.5}$.

      \textbf{Proof of Claim:} One can verify that $\msf H_{k+.5}$
      satisfies all the properties of an ambient
      isotopy.\footnote{Intuitively, all it is doing is sliding all the
        points in $V_{k+1}$ along the lines in \cref{fig:vs-with-lines}
        until they are \emph{either} in $V_{k+.5} \setminus V_{k+1}$
        \emph{or} $\frac{1}{c}$ times as far from $q_k$ as they were at
        the start.} It remains to show that $\msf h_{k+.5} = \msf
      H_{k+.5}(1, \cdot)$ satisfies the conditions stipulated near
      \cref{eq:h1.5-bound}.

      Let $p \in V_{k+.5}$ be arbitrary. We have two cases.
      \begin{enumerate}
        \item Suppose $s(p) \in \pb{s_{c_0}, 1}$. Then $\msf h_{k+.5}(p)
          \in V_{k+.5} \setminus V_{k+1}$, and hence $\msf h_{k+1}(\msf
          h_{k+.5}(p)) \in V_{k} \setminus V_{k+1}$. \cmark
        \item Suppose $s(p) \in \bk{0, s_{c_0}}$. One can verify that in
          this case, $\msf H_{k+.5}(t, p)$ only slides $p$ along a ray
          segment originating from $q_k$, with the sliding dictated by
          $s'(t,p)$. Hence
          \begin{align*}
            \frac{d(p, q_k)}{d(h_{k+.5}(p),\ \msf h_{k+.5}(q_k))}
            &= \frac{d(\msf H_{k+.5}(0, p), H_{k+.5}(0,
              q_k))}{d(H_{k+.5}(1, p), H_{k+.5}(1, q_k))} \\
            &= \frac{s'(0, p)}{s'(1, p)} \\
            &= \frac{\cancel{\frac{s(p)}{s_{c_0}}}
              s_{c_0}}{\cancel{\frac{s(p)}{s_{c_0}}}
              s_{c_1}} \\
            &= c.
          \end{align*}
          Simplifying gives us
          \[
          \frac{1}{c} \cdot d(p, q_k) = d(H_{k+.5}(1, p),\ H_{k+.5}(1,
          q_k)),
          \]
          as desired.
      \end{enumerate}
    \end{leftbar}
  \end{adjustwidth}
  \textbf{Guaranteeing Bijectivity:} Observe that for all $k\in\NN$,
  for all $p \in V_{k+1}$, we have
  \begin{align}
    d(\msf{h}_{k+1}(\msf h_{k+.5}(p)),\ \msf{h}_{k+1}(\msf
    h_{k+.5}(q_{k})))
    &\geq c \cdot d(\msf h_{k+.5} (p),\ \msf h_{k+.5}(q_{k}))
      \nonumber\\
    &\geq c \cdot \frac{1}{c} \cdot d(p, q_{k}) \nonumber\\
    &\geq d(p, q_{k}). \label{eq:final-bound}
  \end{align}
  For each $n \in \NN$, let $\ms h_n$ denote the composition of all
  these homeomorphisms:
  \begin{align*}
    \ms h_n
    &= \pn{\comp_{k=1}^{n-1} \pn{\msf{h}_{k+1} \circ \msf h_{k+.5}}} \circ
      \msf h_1 \\
    &= (\msf h_n \circ \msf h_{n-.5} \circ \cdots \circ \msf h_2 \circ
      \msf h_{1.5} \circ \msf h_1).
  \end{align*}
  Note that the sequence of points $\ms h_n^{-1}(q_n)$ is constant,
  hence the limit $\lim_{n\to\infty} \ms h_n^{-1}(q_{n})$ exists; in
  particular, it is $y_\infty$. For all $y \in \RR^3 \setminus
  \set{y_\infty}$, \cref{eq:final-bound} shows that at each step, $y$
  is sent no closer to $q_{k+1}$ than it was to $q_k$. Since the boxes
  are shrinking it follows that each such $y$ will eventually leave
  the boxes and thus remain fixed at subsequent steps. Explicitly: If
  $n_0$ satisfies
  \[
    \frac{(6+2\varepsilon)\ell}{2^{n_0}} < d(y, y_\infty),
  \]
  Then for all $n > n_0$ we have $\ms h_{n}(y) \not \in V_n$, and
  hence
  \[
    \ms h_n(y) = \ms h_{n_0}(y).
  \]
  This implies $\ms h_\infty$ is a bijection between $\RR^3 \setminus
  \set{y_\infty}$ and $\RR^3 \setminus \set{\ms h_\infty(y_\infty)}$.
  So \cref{thm:vks-ambient-homeomorphism} implies $\ms h_\infty$ is a
  homeomorphism between $\RR^3 \setminus \set{y_\infty}$ and $\RR^3
  \setminus \set{\ms h_\infty(y_\infty)}$. Thus $\ms h_\infty$ is
  bijective on $\RR^3$, and \cref{thm:vks-ambient-isotopy} implies
  that $\ms H_\infty(1, \cdot)$ is an ambient isotopy.
\end{proof}
Finally, we have the following famous example.
\begin{proposition}\label{prop:countable-csum-of-trefoils}
  The following curve is a tame arc.
  \begin{figure}[H]
    \centering
    \includegraphics[scale=.5]{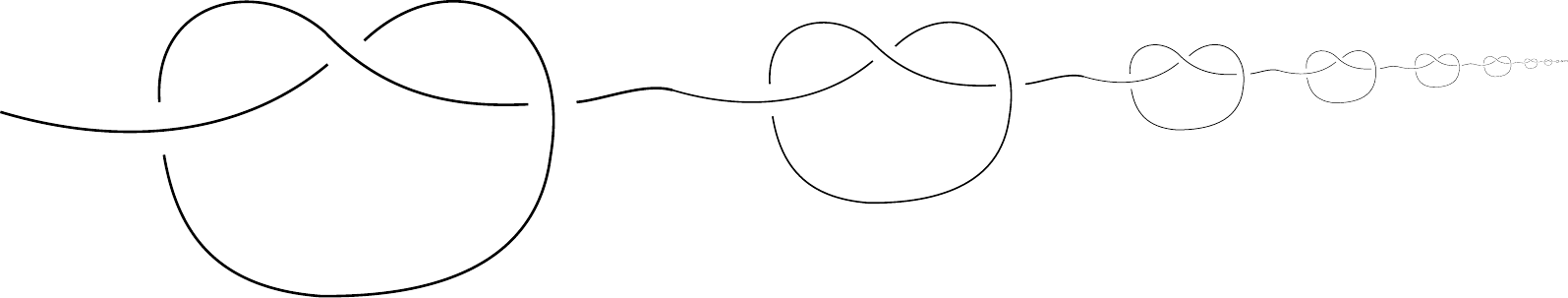}
  \end{figure}
\end{proposition}
\begin{proof}[Sketch]
  We apply \cref{lem:disjoint-vks}. Consider a sequence of properly
  nested boxes $V_1$, $V_2$, $\ldots$, as follows:
  \begin{figure}[H]
    \centering
    \includegraphics[scale=.5]{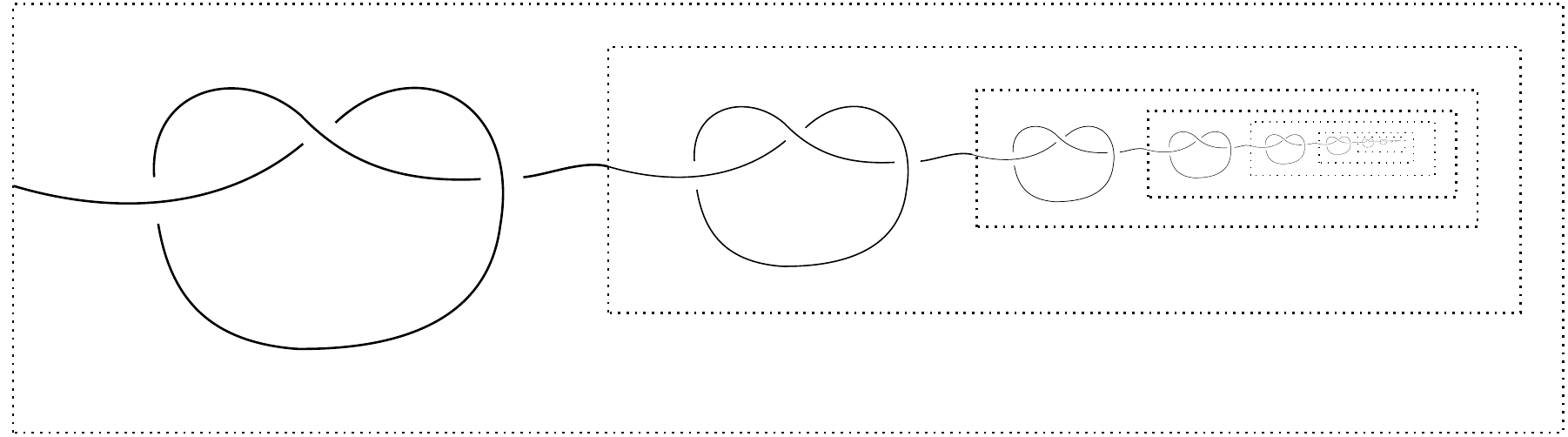}
    \caption{A countable connected sum of trefoils.}
    \label{fig:countable-csum-of-trefoils}
  \end{figure}
  For all $k \in \NN$, define $V'_k$ by $V'_k = V_k \setminus
  V_{k+1}$. Note the $V'_k$ are disjoint. We can define the ambient
  isotopies $H_k$ such that $H_k$ performs the following modification
  on $V'_k$:
  \begin{figure}[H]
    \centering
    \includegraphics[scale=.5]{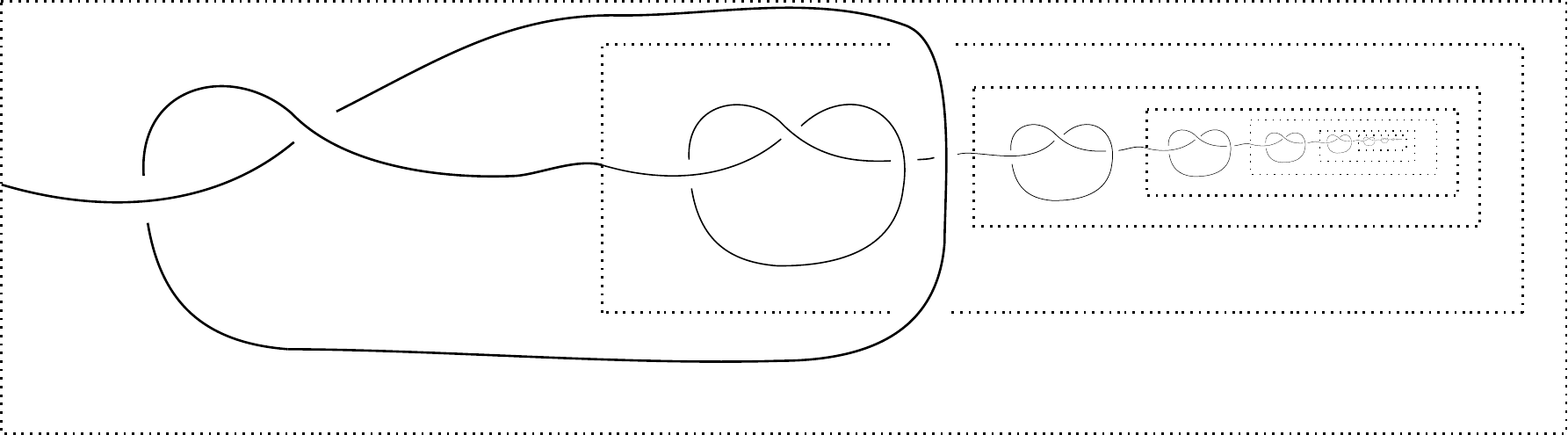}
  \end{figure}
  \begin{figure}[H]
    \centering
    \includegraphics[scale=.5]{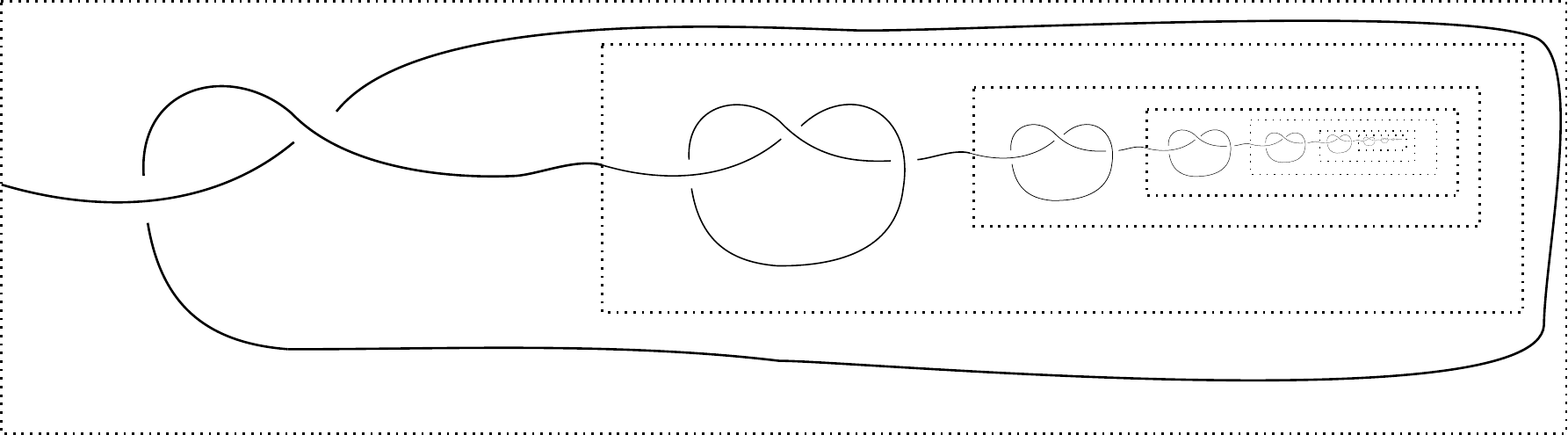}
  \end{figure}
  \begin{figure}[H]
    \centering
    \includegraphics[scale=.5]{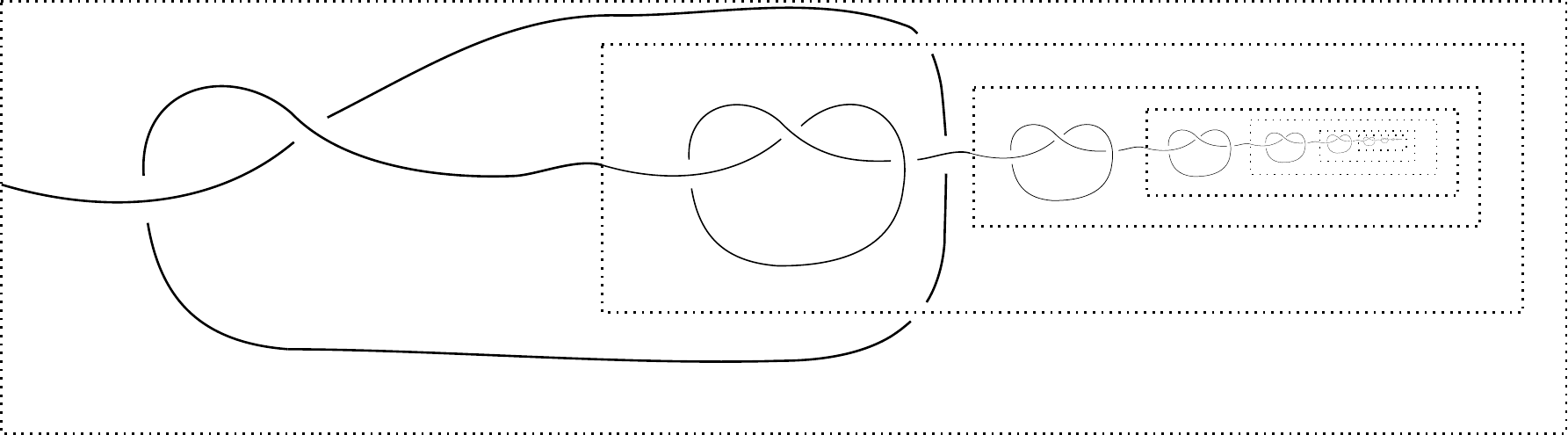}
  \end{figure}
  \begin{figure}[H]
    \centering
    \includegraphics[scale=.5]{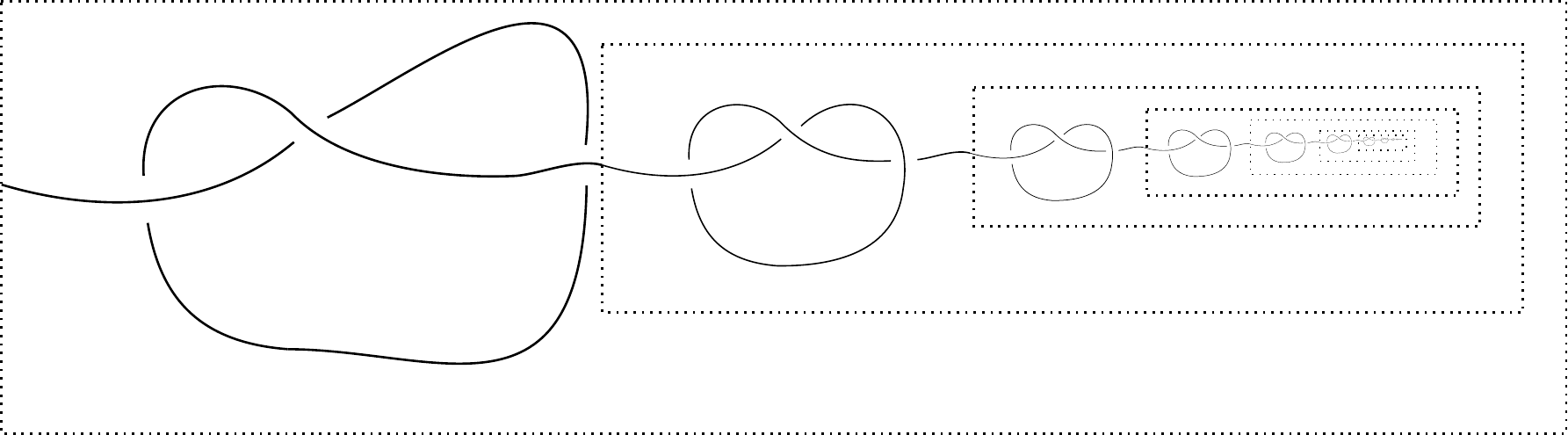}
  \end{figure}
  \begin{figure}[H]
    \centering
    \includegraphics[scale=.5]{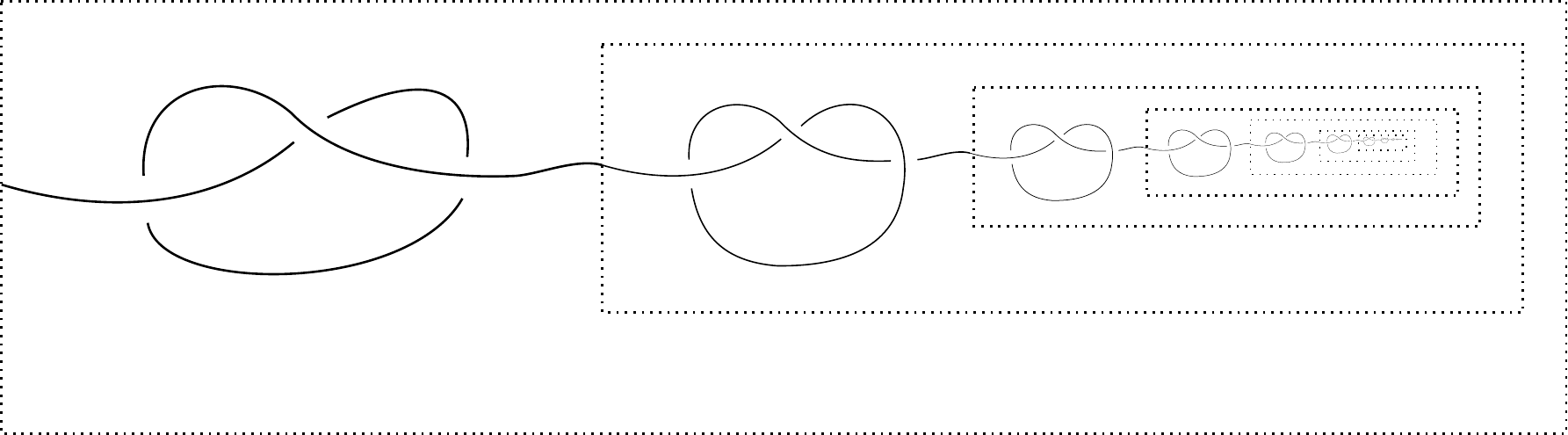}
  \end{figure}
  \begin{figure}[H]
    \centering
    \includegraphics[scale=.5]{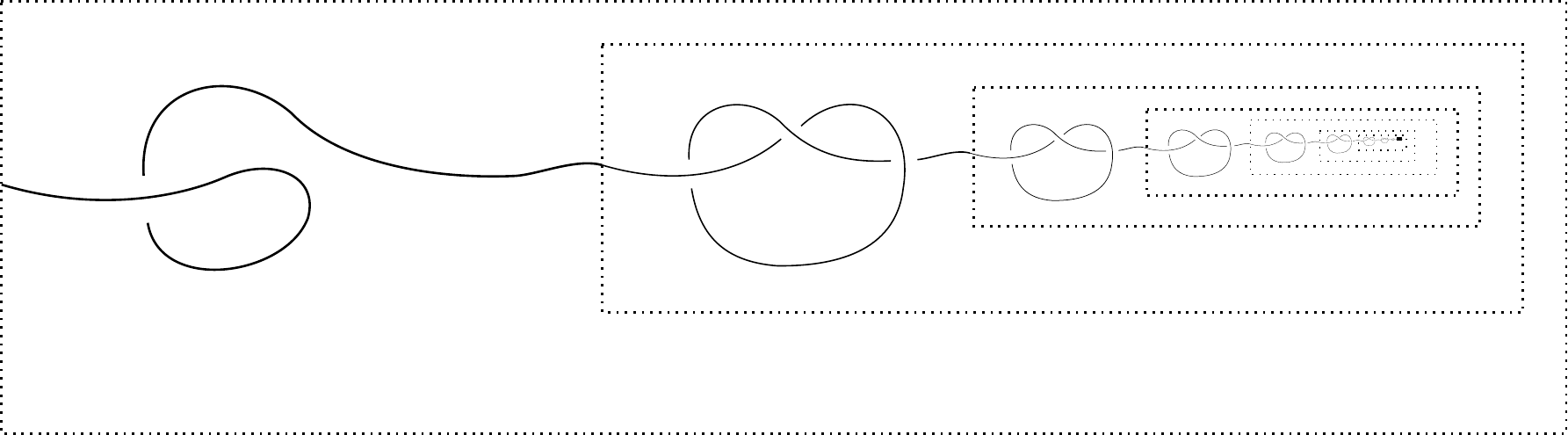}
  \end{figure}
  \begin{figure}[H]
    \centering
    \includegraphics[scale=.5]{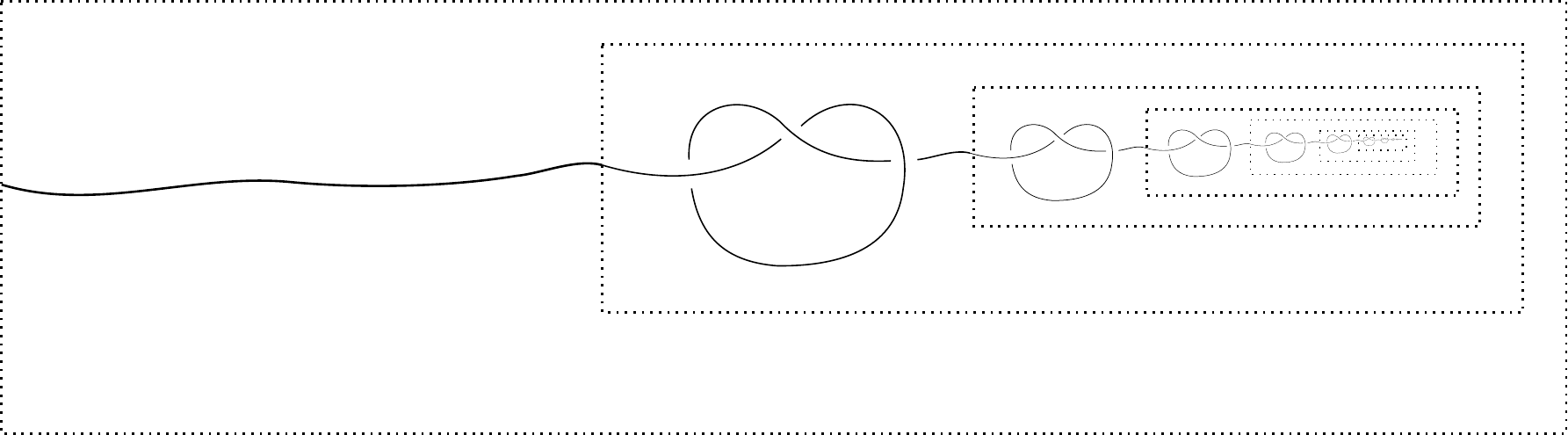}
  \end{figure}
  Taking the limit, we obtain an ambient isotopy unknotting the arc.
\end{proof}
Having established the versatility of \cref{thm:vks-ambient-isotopy},
we now discuss situations in which it cannot be applied. In a sense,
we will see that each of the hypotheses of the theorem are sharp.

\section{Cases Where \cref{thm:vks-ambient-isotopy} Does Not
  Apply}\label{sec:does-not-apply}

\begin{example}
  If we extend the right hand side of
  \cref{fig:countable-csum-of-trefoils} with a straight line segment,
  then we cannot apply \cref{thm:vks-ambient-isotopy}. In fact, the
  curve is wild --- see \cite{Daverman}, Exercise 2.8.4.
  \begin{figure}[H]
    \centering
    \includegraphics[scale=.5]{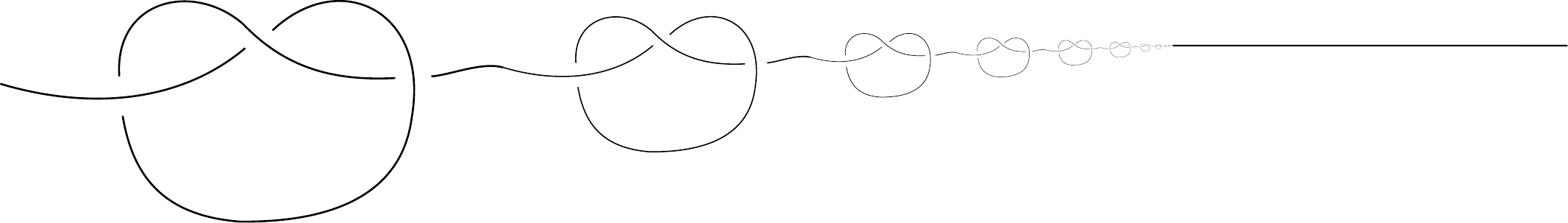}
  \end{figure}
  What breaks here is that if we try to apply the same argument as we
  did with the non-extended version in
  \cref{prop:countable-csum-of-trefoils}, we can't force
  $\lim_{n\to\infty} \diam\pn{\bigcup_{k=n}^\infty V_k} = 0$. In
  particular, $\diam\pn{\bigcup_{k=n}^\infty V_k}$ is bounded below by
  the length of the straight line segment.
\end{example}
And now, as promised, we discuss the curve from
\cref{subfig:remarkable-curve}.
\begin{example}\label{ex:remarkable-curve-redux}
  One cannot apply \cref{thm:vks-ambient-isotopy} to the following
  curve:
  \begin{figure}[H]
    \centering
    \includegraphics[scale=1.5]{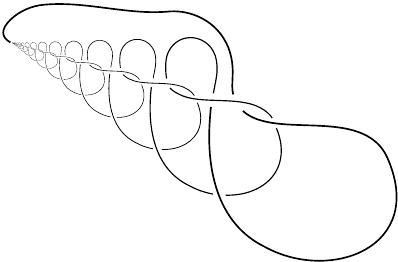}
    \caption{Fox's ``Remarkable Knotted Curve.''}
    \label{fig:remarkable-curve-redux}
  \end{figure}
  Here, the $V_k$'s are not the problem; rather it is bijectivity on
  the ambient space. Consider a ``line'' of points in the ambient
  space passing through the first loop:
  \newcommand{\figscale}{8}
  \begin{figure}[H]
    \centering
    \includegraphics[scale=\figscale]{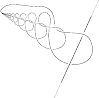}
    \caption{The curve, now with an imaginary ``line'' of points from
      the ambient space.}
    \label{fig:remarkable-curve-with-line}
  \end{figure}
  As we remove the first loop, some points on the ``line'' get pulled
  through:
  \begin{figure}[H]
    \centering
    \includegraphics[scale=\figscale]{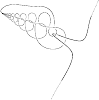}
  \end{figure}
  \begin{figure}[H]
    \centering
    \includegraphics[scale=\figscale]{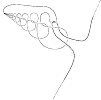}
    \caption{Removing the first loop. Note, by continuity, the
      ``line'' must remain unbroken.}
  \end{figure}
  As we remove the second loop, a similar process occurs:
  \begin{figure}[H]
    \centering
    \includegraphics[scale=\figscale]{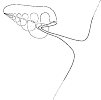}
  \end{figure}
  \begin{figure}[H]
    \centering
    \includegraphics[scale=\figscale]{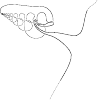}
    \caption{Removing the second loop.}
  \end{figure}
  As $n \to \infty$, the stitching process continues, with the number
  of lines doubling at each iteration. No matter what we try, in the
  limit, a countable subset of the original line gets mapped to the
  wild point. In fact, if we thicken the line in
  \cref{fig:remarkable-curve-with-line} to a cylinder, we see that
  \emph{uncountably}-many points are lost in the limit!

  This is reflected in the fact that if we were to try something like
  the approach taken in \cref{prop:countable-r1-2}, we would not be
  able to define ambient isotopies that do the jobs of the $\msf H_{k
    + .5}$'s.
\end{example}
\begin{remark}\label{rem:cant-tie-it}
  Another perspective one might consider here is that there is no
  obvious way to ``tie'' the curve in
  \cref{fig:remarkable-curve-redux} if starting with just an unknotted
  line. Indeed, there's no ``first loop'' to insert --- to tie one, we
  require infinitely-many of the others to be present already!
\end{remark}

\section{Discussion}
We conclude with a discussion of directions for future work.

\cref{thm:vks-ambient-homeomorphism,thm:vks-ambient-isotopy} give us
one direction to a loose ``countable analogue'' of Reidemeister's
theorem. The restrictions on the $V_k$'s have a nice diagrammatic
interpretation --- ``the total region we're going modify has to shrink
in the limit'' --- but so far, a similarly-concise description of the
bijectivity requirement has eluded the author.

Qualitatively, it seems that problems tend to occur when the set of
{points that get moved infinitely-many times} is not topologically
discrete; however, it's been difficult to find the right language to
distinguish between cases when this gives rise to \emph{legitimate}
problems versus ones where the issue is superficial. As an example of
the former, consider \cref{ex:remarkable-curve-redux}, and as an
example of the latter, consider a sequence of homeomorphisms $h_k :
\RR^3 \to \RR^3$ where each $h_k$ is defined by
\[
  h_k(y) = y +
  \begin{bmatrix}
    \frac{1}{2^k} \\[.5em]
    0 \\
    0
  \end{bmatrix}.
\]
Then
\[
  h_\infty(y) = y +
  \begin{bmatrix}
    1 \\
    0 \\
    0
  \end{bmatrix},
\]
and so \emph{all} points in $\RR^3$ are moved infinitely-many times by
$h_\infty$, yet we have no problems.

Thus, we have the following question:
\begin{adjustwidth}{1em}{}\vspace{.5em}
  \begin{leftbar}\vspace{-.5em}
  \begin{question}
    Is there a simpler way to guarantee bijectivity of the limit
    function in \cref{thm:vks-ambient-isotopy}? In particular, is there
    a purely diagrammatic condition?
  \end{question}
\end{leftbar}
\end{adjustwidth}

One of the things that makes the problem in
\cref{ex:remarkable-curve-redux} tricky to spot at first is that the
limiting process yields an \emph{isotopy}, just not an \emph{ambient
  isotopy} (this is reminiscent of \emph{Bachelor's unknotting}). We
wonder whether a similar effect can be obtained using only
Reidemeister I or Reidemeister III moves, as detailed in the following
questions:
\begin{adjustwidth}{1em}{}\vspace{.5em} % sad
\begin{leftbar}\vspace{-.5em}
\begin{question}\label{q:what-about-r1}
  Does there exist a knot $f \colon S^1 \into \RR^3$ such that applying a
  countable sequence of Reidemeister I moves to $f$ yields an isotopy,
  but not an ambient isotopy?
\end{question}
\begin{question}
  Same as \cref{q:what-about-r1}, but with Reidemeister III moves
  instead of Reidemeister I moves.
\end{question}
\end{leftbar}
\end{adjustwidth}
Of course, we want to avoid trivial examples like taking $H_1$ to be
Bachelor's unknotting and then taking the remaining $H_k$'s to be
identity.

Now, returning to the question of a countable analogue for
Reidemeister's theorem:
\begin{question}
  If we restrict ourselves to Reidemeister moves, when do the
  converses of
  \cref{thm:vks-ambient-homeomorphism,thm:vks-ambient-isotopy} hold?
  That is, given an ambient isotopy between two embeddings $f_0, f_1 :
  S^1 \into \RR^3$, when can we guarantee the existence of the desired
  $V_k$'s and $h_k$'s, where each of the $h_k$'s represent single
  Reidemeister moves?
\end{question}
We have a conjecture in this direction. In \cite{KobayashiThesis} the
author employed \cref{thm:vks-ambient-isotopy} to argue the following
result:
\begin{theorem}
  Call a knot diagram a \textbf{discrete diagram} if it satisfies all
  the axioms of a regular diagram except perhaps having finitely-many
  crossings.\footnote{The ``discrete'' in ``discrete diagram'' comes
    from the fact that the set of crossing points only needs to be
    topologically discrete rather than finite.}

  Now, let $f \colon S^1 \into \RR^3$ be an arbitrary knot. Then if $f$
  admits a discrete diagram, $f$ is ambient isotopic to a
  representative comprised of countably-many line segments.
\end{theorem}
This gave rise to the following conjecture (we thank Kye Shi for the
insight of adding the fourth move):
\begin{conjecture}
  Define the \emph{extended} Reidemeister moves to be the standard
  move set together with a fourth move
  \begin{figure}[H]
    \centering
    \includegraphics{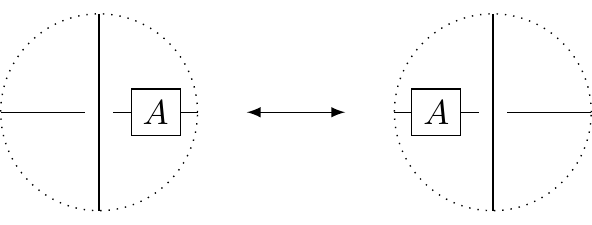}
    \caption{Fourth move}
  \end{figure}
  where in the above, $A$ is a compact set whose interior remains
  fixed relative to its boundary. Note, $A$ can contain wild points.

  Let $f_0, f_1 \colon S^1 \into \RR^3$ be knots that admit discrete
  diagrams. Suppose $f_0 \cong f_1$. Then there exists a countable
  sequence of extended Reidemeister moves satisfying the hypotheses of
  \cref{thm:vks-ambient-isotopy} and taking $f_0$ to $f_1$.
\end{conjecture}
For more details on the proposed approach, see \cite{KobayashiThesis},
\S 9.3.1. We have some partial results in this direction but there
remain important gaps.

\section{Acknowledgements}
We thank Kye Shi for proofreading and for offering many helpful
suggestions on how to improve the exposition. We also thank Francis Su
and Sam Nelson for the many insights they provided when advising the
author's undergraduate thesis \cite{KobayashiThesis}, on which this
work is based.

\bibliographystyle{amsplain}
\bibliography{uniform-convergence-and-knot-equivalence}

\appendix

\section{The Fox-Artin Tameness Invariant}\label{sec:tameness-invariant}
We will now describe the invariant for tameness established by Fox and
Artin.
\begin{theorem}[Fox-Artin]\label{thm:fox-artin-invariant}
  Let $f \colon [0,1] \into \RR^3$ be a tame arc. Let $p = f(0)$, and for
  all $k \in \NN$, let $V_k \subseteq \RR^3$ be a closed set such that
  $p \in V_k^\circ$. Suppose that
  \[
    \cdots \subsetneq V_k \subsetneq \cdots \subsetneq V_2 \subsetneq
    V_1,
  \]
  and
  \[
    \set{p} = \bigcap_{n=1}^\infty V_k.
  \]
  Then there exists $n_0 \in \NN$ such that the inclusion map
  \[
    \iota_* \colon \pi(V_{n_0} \setminus f([0,1])) \into \pi(V_1 \setminus
    f([0,1]))
  \]
  is the trivial homomorphism.
\end{theorem}
\begin{proof}
  $f$ is tame implies that there exists an ambient homeomorphism $h :
  \RR^3 \to \RR^3$ such that $h \circ f$ is a straight line segment.
  Since homeomorphism preserves the fundamental group (as well as all
  the conditions on the $V_n$), it suffices to prove the claim for a
  straight line.

  Hence, without loss of generality, suppose $f$ is a line segment.
  \begin{figure}[H]
    \centering
    \includegraphics{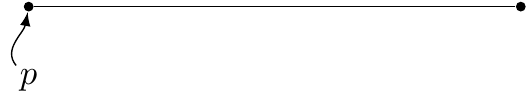}
    \caption{An example of $f([0,1])$, with the choice of $p$ labeled.}
  \end{figure}
  \begin{figure}[H]
    \centering
    \includegraphics[scale=.75]{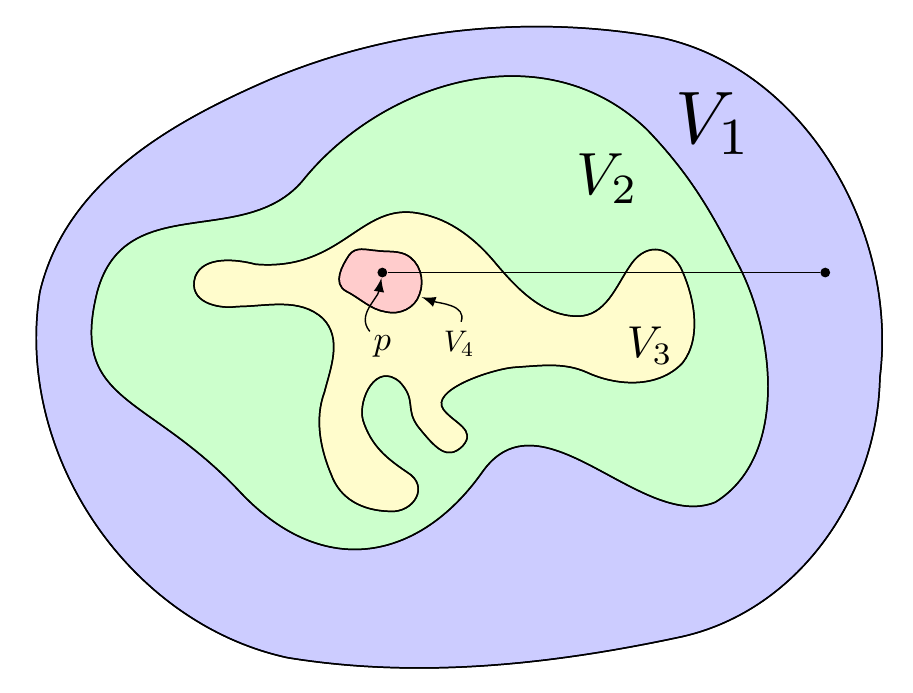}
    \caption{An example of some possible first few $V_n$'s.}
  \end{figure}
  Since $p \in V_1^\circ$, there exists $\varepsilon > 0$ such that
  $B_\varepsilon(p) \subsetneq V_1$. Now, because
  $\bigcap_{n=1}^\infty V_n = \set{p}$, we have $\diam(V_{k}) \to 0$,
  and hence there exists $n_0 \in \NN$ such that $V_{n_0} \subsetneq
  B_\varepsilon(p)$. This gives us the inclusions (of sets)
  \[
    \iota_{0} \colon V_{n_0} \into B_\varepsilon(p)
    \qquad \text{and} \qquad
    \iota_{1} \colon B_\varepsilon(p) \into V_1.
  \]
  Then the inclusion (of sets) $\iota \colon V_{n_0} \into V_1$ is given by
  \[
    \iota = \iota_1 \circ \iota_0.
  \]
  The result is identical if we replace $V_{n_0}$, $B_\varepsilon(p)$,
  and $V_1$ with $V_{n_0} \setminus f([0,1])$, $B_\varepsilon(p)
  \setminus f([0,1])$, and $V_1 \setminus f([0,1])$, respectively.

  Since the inclusion maps are all continuous, they induce
  homomorphisms on the associated fundamental groups. Thus, defining
  \begin{align*}
    (\iota_0)_* &: \pi(V_{n_0} \setminus f([0,1])) \into
                  \pi(B_\varepsilon(p) \setminus f([0,1])) \\
    (\iota_1)_* &: \pi(B_\varepsilon(p) \setminus f([0,1]))
                  \into \pi(V_1 \setminus f([0,1])),
  \end{align*}
  we see that $\iota_* \colon \pi(V_{n_0} \setminus f([0,1])) \into \pi(V_1
  \setminus f([0,1]))$ is given by
  \[
    \iota_* = (\iota_1)_* \circ (\iota_0)_*.
  \]
  $B_\varepsilon(p) \setminus f([0,1])$ is just a ball with a radius
  removed, hence $\pi(B_\varepsilon(p) \setminus f([0,1]))$ is the
  trivial group. It follows that $(\iota_1)_*$ is the trivial
  homomorphism, and thus $\iota_*$ is the trivial homomorphism.
\end{proof}
In the case of the curve in \cref{subfig:remarkable-curve}, one can
find a sequence of $V_k$ such that no such $n_0$ exists. It follows
that the curve is wild.

\end{document}